\documentclass[11pt]{amsart}
\usepackage{amsmath, amsthm, amssymb, amsfonts, verbatim,color}

\usepackage[english]{babel}

\usepackage{wrapfig}

\usepackage{float}

\usepackage{graphics}

\usepackage[pdftex]{graphicx}
\usepackage[bookmarks]{hyperref}
\reversemarginpar            

\definecolor{darkgreen}{rgb}{0,0.55,0}

\newtheorem{theorem}{Theorem}[section]

\newtheorem{lemma}[theorem]{Lemma}
\newtheorem{prop}[theorem]{Proposition}
\theoremstyle{definition}

\theoremstyle{remark}

\numberwithin{equation}{section}
\numberwithin{theorem}{section}

\newcommand{\R}{{\mathbb{R}}}
\newcommand{\tep}{t_0}
\newcommand{\C}{{\mathbb{C}}}
\newcommand{\N}{{\mathbb{N}}}
\newcommand{\uH}{{\underline{H}}}
\newcommand{\uM}{{\underline{M}}}

\newcommand{\dist}{{\operatorname{dist}}}
\newcommand{\supp}{{\operatorname{supp}}}

\newcommand{\calO}{{\mathcal{O}}}

\newcommand{\calN}{{\mathcal{N}}}
\newcommand{\lam}{{\lambda,m}}
\newcommand{\Ktwo}{\mathcal J_2}

\newcommand{\hex }{h_{ex}}

\def\be{\begin{equation}}
\def\ee{\end{equation}}

\def\ep{\epsilon}
\def\vp{\varphi}

\def\calA{\mathcal A}

\begin{document}
\title[Stable configurations with unbounded vorticity for $2$d Ginzburg-Landau]{Local minimizers with unbounded vorticity for the $2$d Ginzburg-Landau functional.}

\author{Andres Contreras \and Robert L. Jerrard  }
\address{Department of Mathematical Sciences, New Mexico State University, Las Cruces, New Mexico, USA}\email{acontre@nmsu.edu}
\address{Department of Mathematics, University of Toronto,
Toronto, Ontario, Canada}\email{rjerrard@math.toronto.edu}

\date{\today}
\maketitle

\begin{abstract}
A central focus of Ginzburg-Landau theory is the understanding 
and characterization of vortex configurations.  On a bounded domain 
$\Omega\subseteq \R^2,$ global minimizers, and critical states in general, of the corresponding energy functional have been studied thoroughly in the limit $\ep\to 0,$ where $\ep>0$ is the inverse of the Ginzburg-Landau parameter. The presence of an applied magnetic field  of strength $\hex\gg 1$ makes possible the existence of stable vortex states. A notable open problem is whether there are solutions of the Ginzburg-Landau equation for any number of vortices below $ h_{ex} \vert \Omega\vert /2 \pi,$ for external fields of up to super-heating field strength. The best earlier partial results give, for every $0<c<1,$ and $K>0,$ the existence of local minimizers of the Ginzburg-Landau functional with a prescribed number of vortices in the range $1 \leq N \leq \min \{ K \vert \log \ep \vert, c ( h_{ex} \vert \Omega\vert /2 \pi ) \}$  and  for values of $1\ll_\ep h_{ex}$  smaller than a power of the Ginzburg-Landau parameter.

In this paper, we prove that there are constants $K_1, \alpha>0$ such that given natural numbers satisfying
\[1\leq N \leq \frac{\hex}{2\pi}(|\Omega|-h_{ex}^{-1/4}),\]
local minimizers of the Ginzburg-Landau functional with this many vortices exist, for fields such that $K_1\leq \hex \leq 1/\ep^{\alpha}.$ Our strategy consists in combining: the minimization over a subset of configurations for which we can obtain a very precise localization of vortices; expansion of the energy in terms of a modified vortex interaction energy that allows for a reduction to a potential theory problem; and a quantitative vortex separation result for admissible configurations. 
Our results provide detailed information about the vorticity and refined asymptotics of the local minimizers that we construct.
\end{abstract}

\section{introduction }

Let $\Omega$ be a bounded, open, simply-connected
subset of $\R^2$ with smooth boundary.
Given $(u,A)\in H^1(\Omega;\C)\times H^1(\Omega;\R^2)$, we define
the Ginzburg-Landau functional
\[
GL_\epsilon(u,A) := \frac 12 \int_\Omega
 |(\nabla - iA)u|^2 + |\nabla\times A - \hex |^2 + \frac {(1-|u|^2)^2}{2\epsilon^2}.
\]
Quantized vortices, described below in detail, are a prominent qualitative feature
of {a large class of} critical points of {$GL_\ep$}, relevant to both
physical phenomena 
and certain problems of a geometric flavour. The influential work \cite{BBH} characterizes minimizers of a simplified version without magnetic field,  where vortices emerge as a result of imposed topologically non-trivial boundary conditions. Later, this work was extended to a problem contemplating magnetic influences in \cite{BR}. In a series of works, starting with \cite{Serf1} and \cite{Serf2}, continuing with \cite{SS1,SS2} the monograph \cite{SandSerfbook} (and references therein) and culminating in \cite{SandSerfCMP}, the vortex structure of global minimizers of the full model has been described in great detail for a wide range of values of the external field $\hex$. It is known that minimizers transition from a vortex-less state to one where a specific number of vortices is preferred, as the external field increases over a threshold called the \textit{first critical field}. On the high end of {strengths} of applied fields considered {in \cite{SandSerfCMP}}, the optimal number of vortices, which diverges as $\ep \to 0,$ and their asymptotic distribution is obtained at main order.

A satisfactory picture of the moduli space of solutions to the Ginzburg-Landau equations should not only characterize global minimizers but also other stable equilibria. In $2$d Ginzburg-Landau, the existence of branches of solutions with a prescribed number of vortices (different from those present in a global minimizer) in a range determined by the capacity of the applied field to contain them, is a known conjecture. This phenomenon is a mathematical manifestation of the expected hysteretic properties of vortex (and vortex-less) configurations as noted in \cite{Sixlectures, Riviere}. Stable vortex states were obtained  in \cite{DuLin} below the first critical field. On the other hand, it was noted as early as \cite{Serf1, Serf2} that local minimizers with a fixed number of vortices exist for applied fields near the first critical field; these results were extended in \cite{Stable} in particular considering fields in a much larger interval $1\ll h_{ex} \ll 1/\ep^{s}, 0<s<1/2 .$ In \cite{SandSerfbook}, the authors obtain for the first time local minimizers with a possibly divergent number of vortices, although $N\ll C \vert \log \ep\vert^{1/2}$ and close to the highest allowed numbers, these solutions exist for a limited range of external fields (smaller than any power of $1/\ep$). In the list of open questions in \cite{SandSerfbook}, it is asked to extend the results about branches of stable solutions in chapter $11,$ to a larger set of choices of numbers of vortices and applied fields. The work \cite{CS} partially addresses this question and proves, in particular, the existence of solutions with vortices up to $N\sim\vert \log \ep \vert$ for fields sufficiently larger than the first critical field. {The ranges obtained in \cite{CS} improve on previous constructions considerably, however they are still far from establishing the folklore problem about local branches of mimizers with prescribed vorticity and in fact they do not cover a noticeable portion of the expected range:
\[K_1\ll\hex \ll \frac 1\ep \mbox{ and }1 \le N \le N^*(\hex),\]
where $K_1$ is some, possibly large, number and $N^*(\hex)$ is the maximum expected number of vortices that can be contained by a field of strength $h_{ex}.$ The maximum allowed vorticity is believed to be $N^*(\hex)=\frac{h_{ex}|\Omega|}{2\pi}$ based on a free boundary problem associated to the corresponding mean field model \cite{ChapRubSchatz, SandSerfCMP,CS} for $N\to \infty$ vortices. The condition $\hex \ll \frac 1\ep$ comes from the knowledge that the Meissner (vortex-less) solution is stable for fields of these strengths; the strength of the field for which the Meissner solution loses its stability is known as super-heating field.
}

Most of the above mentioned works also give information about the location of vortices in terms of a renormalized energy or averaged versions of it. The work \cite{SandSerfCMP} introduces the \textit{Coulombian renormalized energy} and global minimizers studied there and the local minimizers found in \cite{CS} assort their vortices so as to asymptotically minimize this energy. Understanding reduced models for a divergent number of vortices in this and other related equations is of great interest \cite{ChapRubSchatz, KurSpi, JSp3, SerfMF}; in the case of $2$d Ginzburg-Landau this interest is partially motivated by connections to problems of crystallization \cite{SandSerfCMP, PetSerf}. In all instances of problems where one has to deal with very large vorticities, the analysis becomes very technically difficult, and these challenges are partly responsible for the lack of progress in the problem of obtaining stable vortex configurations with very large number of vortices.

{
The constructions of local minimizers in the earlier works \cite{SandSerfbook, Serf1, Serf2, Stable} rely on two elements:
\begin{itemize}
\item Roughly speaking, the energy contribution of a vortex for $E_\ep$ (defined in \eqref{Eep.def} below) is $\pi \vert \log \ep \vert ,$ while the energy associated with interaction between vortices scales like
$O(1)\times(\# \mbox{ of pairs of vortices})= O(N^2) .$ The admissible class of functions is chosen so that the energy contribution to $E_\ep$ due to interactions between vortices is known up to $C \vert \log \ep \vert,$  for some $C>0$ sufficiently small. 
\item An energy lower bound for an $N$ vortex configuration of the form
\begin{equation}
GL_\ep (u,A) \geq a_0 h^2_{ex}+\pi N \vert \log \ep\vert+ a_1 N^2+ a_2 N + error (N),
\label{intro1}\end{equation}
for certain {\em explicit} constants $a_0, a_1, a_2$ depending on $\Omega,
\hex$ and  (in a mild way) on $N$ itself.
\end{itemize}
The idea is that the total vorticity of minimizers in the class can be prescribed because their energy is compatible only with the desired vorticity; this is why accurate knowledge of the error is essential. In \cite{SandSerfbook} $error (N)=o(N^2),$ whence the restriction $N^2 \ll \vert \log \ep \vert . $
In \cite{CS}, the authors exploit
lower bounds involving the Coulombian renormalized energy that follow from 
results and techniques in \cite{SandSerfCMP}. The improved lower bounds yield an expansion \eqref{intro1} (for different constants $a_1, a_2$) where $error (N)=o(N).$ 
Given that the range of $N$'s was to be extended to values much higher than $\vert \log \ep\vert ^{1/2},$ the first element can not be combined with the lower bound to prescribe the vorticity as in \cite{SandSerfbook, Serf1, Serf2, Stable}, although a lower bound for the vorticity is available. Instead, the authors of  \cite{CS} devise a new approach whereby a new admissible class allows to bound the total vorticity from above indirectly. Even then the result can only cover a range of $N$'s where the error does not exceed the cost of a vortex.}

In this paper we develop a new strategy that allows us to find
local minimizers of $GL_\ep$ for much larger numbers $N = N_\ep$ of vortices
and applied magnetic field $\hex = h_{ex,\ep} $.
We will always assume that 
\be\label{scaling1}
0<\ep< \ep_0,  \hspace{8em} K_1 \le  \hex \le k_1 \ep^{-1/4},
\ee
and
\be\label{scaling2}
1\le N \le 
\min
\left\{ \frac{\hex}{2\pi}(|\Omega|-\hex^{-1/4}) , 
k_2\ep^{-1/10}\hex^{-1/5} \right\}
\ee
where the constants, fixed below, depend only on $\Omega .$
(In general we write $k_j, K_j$ to denote small and large constants,  and
we always assume that $k_j\le 1 \le K_j$.)
In particular, for $K_1 \le \hex \le k_2' \ep^{-1/12}$,  the entire range $1\le N \le \frac{\hex}{2\pi}(|\Omega|-\hex^{-1/4})$ is included. 
No technical adaptation of earlier arguments seems likely to be of use in this {whole}
range. The key new elements in our approach are:

\begin{itemize}
\item[ 1.] The set over which we minimize prescribes the number of vortices directly: we work with functions $u$ whose vorticity (see \eqref{def.vorticity} below) is close to $\pi \sum_{i=1}^N \delta_{a_i}$
where $a = (a_1,\ldots, a_n)$  is an approximate constrained minimizer of a renormalized energy $H_\ep^N$, defined in \eqref{Hep.newdef}.
\item[ 2.] We derive lower bounds in terms of the renormalized energy.
A similar renormalized energy has appeared before in \cite{SandSerfbook}.
Here, drawing on \cite{JSp2}, we rigorously justify the renormalized energy for very large values of $N$ and $h_{ex}.$ A drawback of these expressions is that the renormalized energy $H_\ep^N$ tends to $-\infty$ as vortices approach the boundary. (In particular, $H_\ep^N$ does not attain its infimum.) 
Moreover,
it loses accuracy as vortices approach the boundary or as any pair of vortices approach 
each other. These considerations give rise to the constraints to which we have alluded above
on the configurations $(a_1,\ldots, a_n)$ that we consider.
\item[ 3.] 
A major advantage of our approach is that, unlike earlier works, 
we do not require\footnote{although such an expansion could presumably be
derived {\em a posteriori} from our results.} an energy expansion of the form \eqref{intro1}. 
But to handle the difficulties mentioned above,
we need {\em a priori} lower bounds for $\min_i \dist(a_i,\partial \Omega)$ and
$\min_{i\ne j}|a_i-a_j|$, when $a= (a_1, \ldots, a_n)$ is an approximate constrained
minimizer of $H_\ep^N$. We also need to show that the ``vorticity
close to $\pi \sum_{i=1}^N \delta_{a_i}$" condition in point 1. above can be improved
for $(u,A)$ minimizing $GL_\ep$ in the admissible class.
To do these, we introduce a modification of the renormalized energy that let us obtain minimizers of this energy in terms of an obstacle problem. From here we can study deviations in almost optimal configurations via a ``screened'' problem. We do this by means of a quantitative version of an argument used in \cite{RNS}, for a similar problem.
\end{itemize}

\subsection{Main result}
To formulate our results, we need some definitions. 
First, let
$G  = G(x,y)$ be the Green's function defined by
\begin{equation}
-\Delta_xG + G =  \delta_y\mbox{ for }x\in\Omega,\qquad\qquad
G(x,y)= 0\mbox{ for }x\in \partial \Omega.
\label{G.def}\end{equation}
We let $S(\cdot, \cdot)$  denote the regular part of $G$, defined by
\begin{equation}
S(x,y) =  2\pi G(x,y) + \log|x-y|.
\label{S.def}\end{equation}
We define
\begin{align}
E_\ep(u) &:= \int_{\Omega} \frac{|\nabla u|^2}2 + \frac{(|u|^2-1)^2}{4\ep^2},
\label{Eep.def}
\\
F(\xi) &:= \frac 12 \int_\Omega |\nabla\xi|^2 + (\xi + 1)^2\,dx \, .
\label{F.def}
\end{align}
We will always write $\xi_0$ to denote the (unique) minimizer of $F$ in $H^1_0(\Omega)$.
Thus $\xi_0$ satisfies
\be\label{xi0.def}
(-\Delta+1)\xi_0 = -1\mbox{ in }\Omega, \qquad \xi_0 = 0\mbox{ on }\partial \Omega.
\ee
For $N\in \N$ and $a = (a_1,\ldots, a_N)\in \Omega^N$, we define the renormalized energy 
\begin{equation}
H_\ep^N(a) =   \sum_{i=1}^N \big[2\pi \hex \xi_0(a_i)  + \pi S(a_i,a_i) \big]+ 2\pi^2 \sum_{i\ne j} G(a_i,a_j),
\label{Hep.newdef}\end{equation}
where we set $G(a,a)=+\infty$ for $a\in \Omega$.
We will see that this approximately characterizes the least possible energy of a pair $(u,A)$ with
vortices near $(a_1,\ldots, a_N)$, up to a constant that depends on $\ep, N$ and $\hex$
but not on vortex locations.

We will always restrict our attention to pairs
$(u,A)$ such that
\be
\nabla \cdot A = 0 \mbox{ in }\Omega, \qquad A\cdot \nu = 0\mbox{ on }\partial \Omega.
\label{coulomb}\ee
holds; in view of the basic gauge invariance property of the Ginzburg-Landau
functional (see for example \cite{SandSerfbook},  Section 2.1.3) , this does not
entail any loss of generality.
Recalling that $\Omega$ is simply connected, we can then write 
\begin{equation}
A = \nabla^\perp B\quad\quad\mbox{ for some (unique) }B\in H^2\cap H^1_0(\Omega) \, .
\label{B.def}\end{equation}
We will write $B = (\nabla^\perp)^{-1}A$ when \eqref{B.def} holds, without indicating the role of the boundary
conditions.

Define
\begin{align}
M_{\ep, N}
&:= \{ a = (a_1,\ldots, a_N) \in \Omega^N : \dist(a_i, \partial \Omega) \ge  \hex^{-1/3}  \},
\label{Mepn.def}
\\
M^*_{\ep, N}
&:= \{ a = (a_1,\ldots, a_N)\in M_{\ep, N} : H_\ep^N(a) \le \min_{M_{\ep,N}} H_\ep^N + \tep\}\ .
\label{Mepnstar.def}\end{align}
for some positive $\tep$, to be chosen below (in the proof of Lemma \ref{L.uMepNstar}).
Thus, configurations in $M^*_{\ep ,N}$ nearly minimize the renormalized energy 
$H_\ep^N$, subject to the constraint that no $a_i$ is too close to $\partial \Omega$.

Given $(u,A)$ satisfying \eqref{coulomb},
writing  $u = u^1+iu^2 \cong (u^1,u^2)$, 
we define the associated {\em vorticity}, denoted $Ju$, by
\be\label{def.vorticity} 
 Ju:= \det \nabla u = \partial_1 u^1\partial_2 u^2-\partial_2 u^1\partial_1 u^2 .
\ee 

A pair $(u,A)$ is thus interpreted as having vortices near $a\in \Omega^N$ if
\be
A\mbox{ satisfies }\eqref{coulomb},  \ \ 
\mbox{ and }\   
\| Ju - \pi \sum_{i=1}^N \delta_{a_i}\|_{\dot W^{-1,1}}\le \sigma_\ep
\label{wherearethey}
\ee
for some small $\sigma_\ep$. 
This means roughly that the vorticity  $Ju$ is strongly
concentrated in a union $\cup_{i=1}^N B_{r_i}(a_i)$ with $\sum r_i\le\sigma_\ep$,
and with $\int_{B_{r_i}(a_i)}Ju \approx \pi$ for every $i$.

In particular, we are interested in the set of pairs $(u,A)$ with vortices
very close to a configuration $a = (a_1,\ldots, a_N)$
that is a near-minimizer of $H_\ep^N$. Thus,  we define
\be\label{calA.def}
\calA_\ep^N := \{ (u,A)\in H^1(\Omega;\C)\times H^1(\Omega;\R^2) :  \ \exists\, a\in M_{\ep,N}^*\mbox{ such that  \eqref{wherearethey} holds} \}
\ee
for a choice of $\sigma_\ep$ to be specified later; see \eqref{sigmaep.def}.
It is a standard fact \footnote{Take a minimizing sequence in $\calA_\ep^N$. By using exactly the
argument to prove existence of unconstrained minimizers for 2d Ginzburg-Landau 
(see for example Proposition 3.5 in section 3.1.5 in \cite{SandSerfbook}), one can extract a subsequence that converges weakly in $H^1\times
H^1$ to a limit, with the energy of the limit bounded by
$\inf_{\calA_\ep^N}GL_\ep$. We thus only need to prove that
the limit belongs to $\calA_\ep^N$, and this follows from weak
continuity properties of the Jacobian, together with the fact that
$M_{\ep, N}^*$ is a closed set.}
 that  $GL_\ep$ attains its minimum in $\calA_\ep^N .$ 
Our main result is

\begin{theorem}
Assume that \eqref{scaling1}, \eqref{scaling2} hold, and let
$(u_\ep, A_\ep)$ minimize $GL_\ep$ in $\calA^N_\ep$,
for $\sigma_\ep$ defined in \eqref{sigmaep.def},
which in particular implies that $\sigma_\ep \le C\ep^{1/2}$.

Then 
$(u_\ep, A_\ep)$ belongs to the interior of $\calA^N_\ep$.
As a result is a  $(u_\ep, A_\ep)$ is a local minimizer of $GL_\ep$,
and hence a solution of the Ginzburg-Landau equations.
\label{T.main}\end{theorem}

This shows both that there exists a local minimizer with $N$ vortices, and
that the vortices are located near points found by minimizing the renormalized  energy.

Theorem \ref{T.main} is the first result showing the existence of local minimizers of $GL_\ep$ with a number of vortices much larger than $\vert \log \ep \vert ,$ going all the way up to $\ep^{-\alpha}$ for some positive $\alpha .$ For fields smaller than $\ep^{-1/12},$ our result settles the conjecture about local minimizers covering the full range of $N$'s. Above this range, that is for $\ep^{-1/12}<h_{ex}<\ep^{-1/4},$ Theorem \ref{T.main} also greatly extends the previous best known partial result \cite{CS} in two directions: the strength of the field is allowed to be as large as $\ep^{-1/4}$ which is much larger than the $\ep^{-1/7}$ in \cite{CS}, and also the number of vortices is still allowed to get as big as $\ep^{-1/12}$ in this range of fields.

Our proof also yields a great deal of information about the local minimizers that we construct.
We show that their vortices are approximated with extreme precision by sums of point
masses at points that {asymptotically minimize $H_\ep^N$}. We also describe {their energy up to errors of order $o(1).$}  For our local minimizers, our results would in principle make it possible to derive explicit estimates in terms of the Coulombian renormalized energy by directly studying the simple discrete energy $H_\ep^N ;$ this would allow to bypass the delicate mass displacement results used in \cite{SandSerfCMP}. We believe that our results may also have some implications for global minimizers, at least when $\hex$ is not too large, but we do not explore that here.

\begin{minipage}{0.45\linewidth}
\begin{figure}[H]
\centering
\includegraphics[trim = 0mm 0mm 0mm 0mm, clip, width = 7cm]{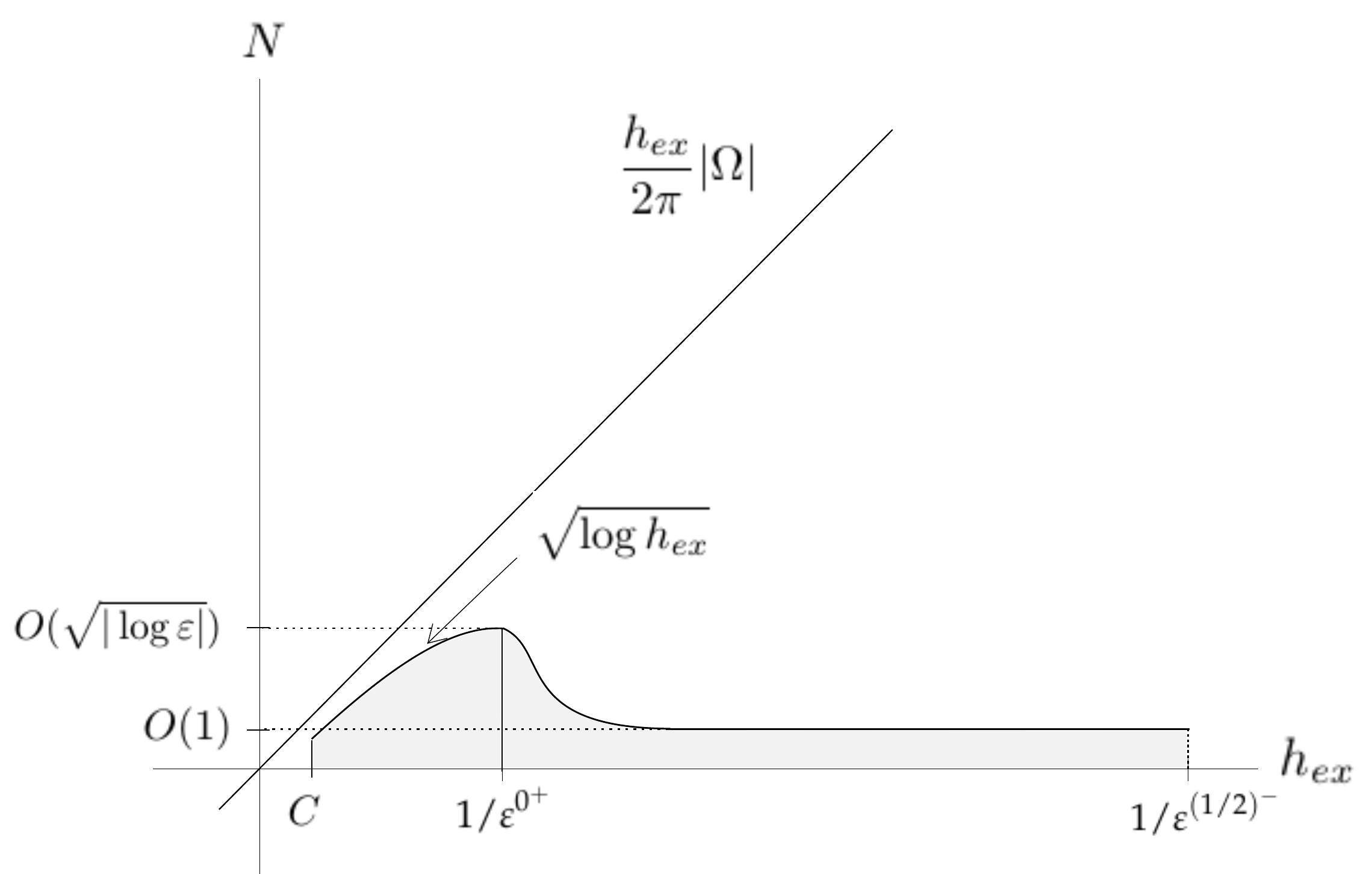}
\caption{\footnotesize Ranges of $N$ and $h_{ex}$ covered in \cite{SandSerfbook, Serf1, Serf2, Stable}.}
\end{figure}
\end{minipage}
\begin{minipage}{0.5\linewidth}
\begin{figure}[H]
\centering
\includegraphics[trim = 0mm 0mm 0mm 0mm, clip, width = 7cm]{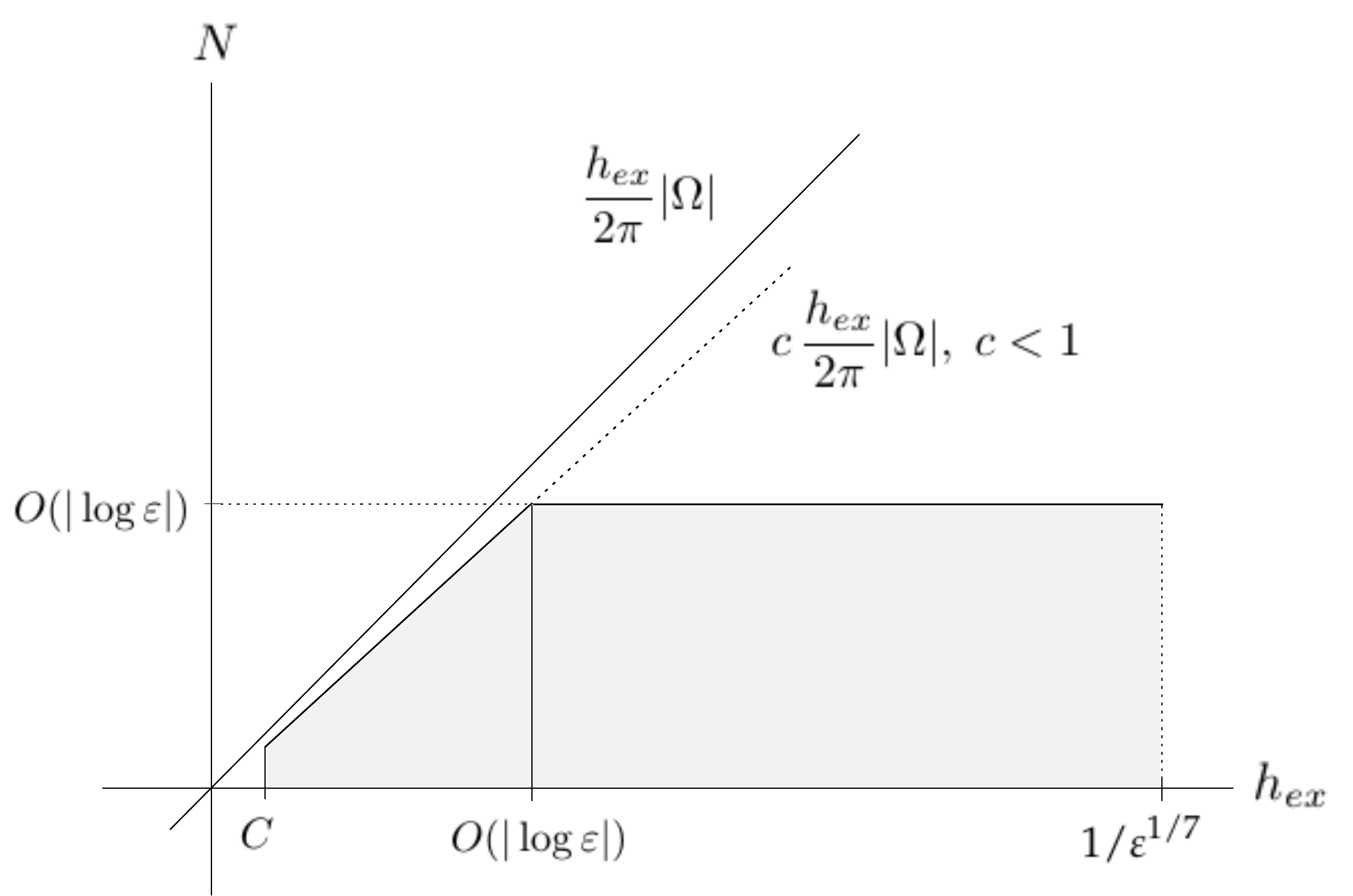}
\caption{\footnotesize Ranges covered in \cite{CS}.}
\end{figure}
\end{minipage}

\begin{figure}[H]
\centering
\includegraphics[trim = 0mm 0mm 0mm 0mm, clip, width = 7cm]{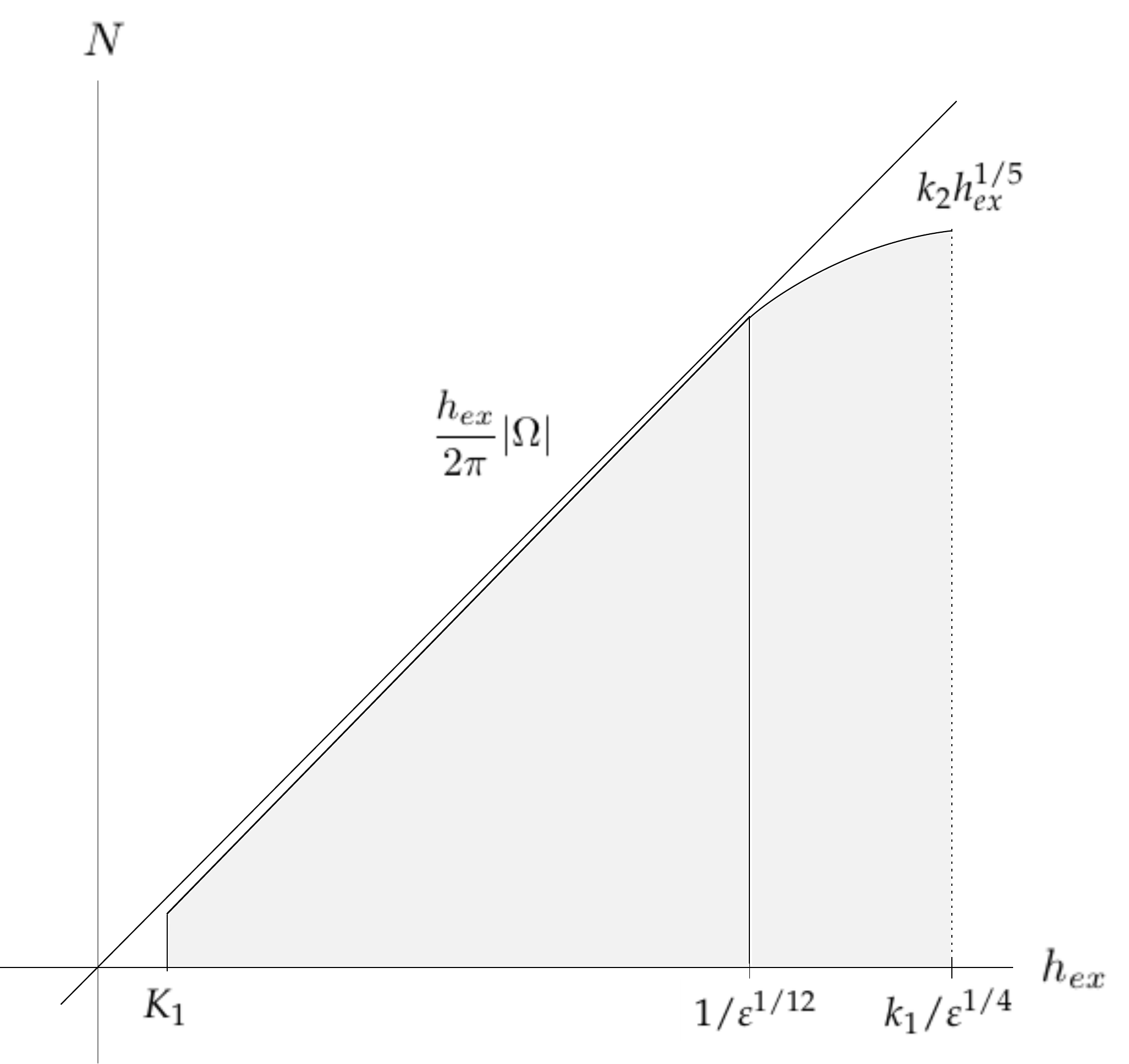}
\caption{\footnotesize Vorticity and external fields in Theorem \ref{T.main}.}
\end{figure}

The organization of the paper is as follows. First, in Section \ref{S-T.main} we present the proof of Theorem \ref{T.main}, assuming various facts that are proved in the remainder of the paper.
 In Section \ref{smp} we introduce a modification of $H_\ep^N$ and study properties of near minimizers of this modification via an auxiliary screened problem. The localization results and corresponding lower bounds are proved in Section \ref{sec:3}. Finally, in Section \ref{EM.Localization:UB, LB}, the upper and lower bounds for minimizers in $\calA^N_\ep$ together with the improved localization of vortices are collected, concluding the proof of Theorem \ref{T.main}.
\vskip.3in
{\it \bf Acknowledgments}
The work of A.C was partially supported by was partially supported by a grant from the Simons
Foundation \# 426318.  The work of R.J. was partially supported by the Natural Sciences and Engineering Research Council of Canada under operating Grant 261955. A.C. wishes to thank S. Serfaty for useful discussions.

\section{Proof of Theorem \ref{T.main}}\label{S-T.main}
In this section we first describe the ingredients in our analysis, and we then
show how these elements combine naturally to yield the proof of our main result.
In doing so, we give a more detailed account of the overall strategy.

\subsection{Ingredients in the proof}

\subsubsection{Interior near-minimizers of the renormalized energy}

The following result
provides information about points in  $M_{\ep,N}^*$, that is,
near-minimizers $a=(a_1,\ldots, a_N)$ of $H_\ep^N$, subject to
the constraint that every $a_j$ stays a certain distance from $\partial \Omega$.
\begin{prop}
Assume that \eqref{scaling1} holds and that
\be\label{scaling2a}
1\le N \le 
\frac{\hex}{2\pi}(|\Omega|-\hex^{-1/4}) 
\ee
(which is implied by \eqref{scaling2}).
Then there exists  $c_0, c_1>0$, depending on
$\Omega$,
such that 
every 
$a = (a_1,\ldots, a_N)\in M^*_{\ep, N}$ satisfies
\begin{align}
\dist(a_i,\partial \Omega)& \ge  c_0\hex^{-1/4} &\mbox{ for all i},
\label{mp.contain}\\
|a_i - a_j| &\ge  c_1\hex^{-1/2}&\mbox{ for all }i\ne j.
\label{mp.sep}
\end{align}
\label{main.prop}\end{prop}

The proof, which we present in Section \ref{smp},  uses ideas from \cite{SandSerfbook, RNS}.

As mentioned earlier the full conjecture about stable vortex states with prescribed vorticity asks to show that the conclusions of
Theorem \ref{T.main} hold if assumption \eqref{scaling2} is replaced by
assumption \eqref{scaling2a}, as long as $K_1\le \hex \ll \frac 1 \ep$.
For use toward a possible proof of
this conjecture, Proposition \ref{main.prop} is sharp. The stronger requirement
\eqref{scaling2} arises from other parts of the proof, described below,
involving upper and lower energy bounds in terms of the renormalized energy.


\subsubsection{Lower energy bounds}

In Proposition \ref{P.gstab}, we prove some results relating the
Ginzburg-Landau energy and the renormalized energy. We show that if
$(u,A)$ satisfies \eqref{wherearethey}, then
\be\label{pmt1}
GL_\ep(u,A) \ge H_\ep^N(a) +\kappa_\ep^{GL}  - \mbox{ error terms}
\ee
where $\kappa_\ep^{GL}$ is a constant defined in \eqref{kappaGL}.
The error terms are quite complicated and depend on $\ep, E_\ep(u), \hex, \sigma_\ep, N, \rho_a$,
where
\be
\rho_a = \frac 14 \min\{ \min_{i\ne j} |a_i - a_j| ,\  \min_i \dist(a_i, \partial \Omega)\}.
\label{rhoa.def}\ee
but will end up being small 
under assumptions \eqref{scaling1}, \eqref{scaling2}, and for 
our eventual choice of $\sigma_\ep$. The
proof of this smallness uses the lower bound for $\rho_a$
that follows from Proposition \ref{main.prop}.
Estimate \eqref{pmt1} is reasonably sharp in the sense that for
every $a\in \Omega^N$ such that $\rho_a$ is not too small, there exist
$(u,A)$ such that \eqref{wherearethey} holds and in addition
\be\label{pmt2}
GL_\ep(u,A) \le H_\ep^N(a) +\kappa_\ep^{GL}  + \mbox{ error terms}
\ee
for error terms of a similar character, and that are similarly small
under our assumptions. 
This follows from arguments in the proof of Proposition \ref{interior}.

These results are adaptations to our setting of estimates
proved in \cite{JSp2}, which dealt with the simplfied functional $E_\ep$,
without magnetic field, rather than the full Ginzburg-Landau 
functional $GL_\ep$. Similar results are proved in \cite{KSp}.
Bounds related to \eqref{pmt1}, \eqref{pmt2} can also be found
in \cite{SandSerfbook}; see for example the formal discussion
leading up to equation (9.3), or the rigorous derivation
(10.2),  which
applies for a bounded number of vortices in the limit as $\ep\to 0$.

\subsubsection{Localization}

The next input needed for Theorem \ref{T.main} is given in Proposition \ref{P.localization}, also adapted from \cite{JSp2}.
It involves the quantity
\[
\Sigma_\ep^{GL}(u,A,a) := GL_\ep(u,A) - \left( H_\ep^N(a) +\kappa_\ep^{GL}\right) ,
\]
which measures the {\em excess} energy of $(u,A)$,
relative to the lower bound \eqref{pmt1}.
The proposition
shows that if $(u,A)$ satisfies \eqref{wherearethey} and $\Sigma_\ep^{GL}(u,A,a)$
is small, then one can find
$\xi = (\xi_1,\ldots, \xi_N)$ near $a$ such that 
\be\label{pmt3}
\| Ju - \pi \sum_{i=1}^N \delta_{\xi_i}\|_{\dot W^{-1,1}}\le \mbox{error terms}
\ee
where the (complicated) error terms depend on the same parameters as \eqref{pmt1},
together with $\Sigma_\ep^{GL}(u,A,a) $. This is a good estimate when the
right-hand side is smaller than $\sigma_\ep$, appearing in hypotheses \eqref{wherearethey}; otherwise it is obvious.

\subsection{Proof of Theorem \ref{T.main}}\label{sec:PMT}

We now describe the proof of our main result.  The inequalities appearing in the
argument are all established in Proposition \ref{interior}

Let $(u_\ep, A_\ep)$ minimize $GL_\ep$ in $\calA_\ep^N$, 
where the parameter $\sigma_\ep$ in
the definition of $\calA_\ep^N$  is in the range
$\ep^{99/100}\lesssim \sigma_\ep \lesssim \ep^{49/100}$.
The precise choice will depend on
$\hex, N$ and $\ep$, see \eqref{sigmaep.def}. 

We first verify, by construction of a competitor, that
\be\label{pmt1a}
GL_\ep(u_\ep, A_\ep) \le \min_{M_{\ep,N}} H_\ep^N + \kappa_\ep^{GL} + \frac \tep 3
\ee
whenever $\ep$ is small enough.
This is an instance of \eqref{pmt2} and its proof essentially contains
that of the general case (which we omit).
We also show that $E_\ep(u_\ep)\le C \hex^2$, which is
needed to make effective use of the lower bound and localization results.

The definition of $\calA_\ep^N$ implies that there exists some $a_\ep\in M_{\ep, N}^*$ such that $(u_\ep, A_\ep)$ and $a_\ep$
satisfy \eqref{wherearethey}. Then \eqref{pmt1a} immediately yields
\[
\Sigma^{GL}_\ep(u_\ep, A_\ep, a_\ep) = 
GL_\ep(u_\ep, A_\ep)  - H_\ep^N(a_\ep)  - \kappa_\ep^{GL}  \le \frac \tep 3.
\]
Note also that Proposition \ref{main.prop} provides a lower
bound $\rho_{a_\ep}\ge c_1\hex^{-1/2}$. 
These estimates and the scaling assumptions \eqref{scaling1}, \eqref{scaling2}  allow us to 
control the error terms in \eqref{pmt3} and
finalize the choice of $\sigma_\ep$ in such a way that
\be\label{pmt2a}
\| Ju_\ep - \pi \sum \delta_{\xi_i}\|_{\dot W^{-1,1}} \le \frac 12 \sigma_\ep
\ee
for some $\xi\in M_{\ep, N}$.
Once this is known, we can apply \eqref{pmt1} to relate $GL_\ep(u_\ep, A_\ep)$ to $H_\ep^N(\xi)$. After controlling error terms as above, this
yields
\[
GL_\ep(u_\ep, A_\ep)  \ge  H_\ep^N(\xi)  + \kappa_\ep^{GL}  - \frac \tep3.
\]
Recalling \eqref{pmt1a}, we deduce that 
\be\label{pmt4a}
H_\ep^N(\xi) \le  \min_{M_{\ep,N}} H_\ep^N + \frac 23 \tep.
\ee
Thus $\xi\in M_{\ep,N}^*$, and in fact
Proposition \ref{main.prop} 
guarantees
that  $\xi\in (M_{\ep,N}^*)^{int}$. 
Then \eqref{mp.contain}, \eqref{mp.sep},\eqref{pmt2a}and \eqref{pmt4a} imply that $(u_\ep, A_\ep)\in (\calA_\ep^N)^{int}$, and is thus a local minimizer.  
This completes the proof. $\hfill\Box$


\section{ Near-minimizers of $H_\ep^N$ in $M_{\ep, N}$.}\label{smp}

In this section we prove Proposition \ref{main.prop}. The crucial idea is to transform this problem into a local argument via a screening process. This screening is made possible by first identifying the leading order distribution of vortices through an obstacle problem. In attempting to carry this out we encounter a nontrivial technical challenge; the renormalized energy $H_\ep^N$ is not bounded from below in $\Omega^N$ and this makes impossible a dual formulation. To overcome this difficulty, we modify the renormalized energy near the boundary so as to have the desired dual formulation, and to do this we need to estimate how fast the divergent parts of $H_\ep^N$ go to $-\infty$ as some of the vortices approach $\partial \Omega .$

We remark that $H_\ep^N$ depends on $\ep$ only through $\hex$.
Similarly, all quantities to be introduced in this 
section, (such as auxiliary functions 
$v_\ep, w_\ep = (-\Delta+1)v_\ep$, \ldots) that appear to depend on $\ep$
in fact depend only on $\hex$ (which however may depend on $\ep$.)
Thus the right hypothesis in these results is not that $\ep$ be sufficiently small,
but rather that $\hex$ be sufficiently large (which however forces $\ep$
to be rather small, in view of \eqref{scaling1}).

\subsection{Modification of $H_\ep^N$ }
Proposition \ref{main.prop} deals with near-minimizers $a = (a_1,\ldots, a_N)$ of $H_\ep^N$,
which is unbounded below, 
subject to a constraint that $\mbox{dist}(a_j,\partial \Omega)$ is not too small.
Our first lemma will allow us instead to
analyze  unconstrained near-minimizers of a function $\uH_\ep^N$ that is continuous
on $\bar \Omega^N$ and in particular bounded below.

\begin{lemma}
There exists a  function $v_\ep\in C^\infty_c(\Omega)$ 
such that
\begin{align}
\|(-\Delta+1)v_\ep \|_{L^\infty(\Omega)} &\le  C \hex^{-1/3} \, |\log \hex| \label{vep1} \ , \\
v_\ep(x) &= \frac 1{2 \hex} S(x,x) \quad\mbox{ if }\mbox{dist}(a,\partial \Omega)\ge \hex^{-1/3}.\label{vep2}
\end{align}\label{L.vep}
\end{lemma}

The lemma will allow us to prove Proposition \ref{main.prop} by studying
near-minimizers of
\begin{equation}\label{uH.def}
\uH_\ep^N(a) : = \sum_{i=1}^N  2\pi \hex \left[ \xi_0(a_i) + v_\ep(a_i)\right]
+ 2\pi^2 \sum_{i\ne j} G(a_i,a_j),
\end{equation}
which coincides with $H_\ep^N$ in $M_{\ep,N}$ as a result of \eqref{vep2}.

\begin{proof}
Let $\chi_\ep:\Omega\to [0,1]$ denote a smooth function such that 
\begin{equation}\label{chiep.def}
\chi_\ep(x) = 1 \mbox{ if }\dist(x,\partial \Omega) \ge \hex^{-1/3},\qquad
\chi_\ep(x) = 0 \mbox{ if }\dist(x,\partial \Omega) \le \frac 12 \hex^{-1/3},
\end{equation}
and  
\begin{equation}\label{Dchi}
\|\nabla \chi_\ep\|_\infty \le C \hex^{1/3}, \qquad \|\nabla^2\chi_\ep \|_\infty \le C \hex^{2/3}.
\end{equation}
We define
\[
v_\ep (x) :=  \frac {\chi_\ep(x)}{2\hex} S(x,x).
\]
Then \eqref{vep2} is immediate. 
For the proof of \eqref{vep1} we will write $s(x) = S(x,x)$.
In view of  \eqref{chiep.def} and \eqref{Dchi},
it suffices to show that
\begin{equation}\label{s.ests}
|s(x)| \le C(|\log d(x)|+1) , 
\qquad
|\nabla s(x)| \le Cd(x)^{-1} , 
\qquad 
|\Delta s(x)| \le Cd(x)^{-2} , 
\end{equation}
for $d(x) = \dist(x,\partial \Omega)$.

\medskip

{\em Estimate of $|s(x)|$}. Let
\begin{equation}\label{splitG}
\tilde S(x,y) =  2\pi \big( G(y,x) -  \Ktwo(y-x) \big)
\end{equation}
where $\Ktwo\in L^2(\R^2)$ is the Bessel potential of order $2$, that is, the unique
square-integrable function on $\R^2$ solving $(-\Delta+1)\Ktwo = \delta_0$.
Well-known properties of $\Ktwo$ include the fact that it is smooth
away from the origin, radial,
with exponential decay as $|x|\to \infty$;
see for example \cite{Bessel1}, 
where  rather explicit formulas may be found. This formula,
or the maximum principle, implies that $\Ktwo\ge 0$.
Also, 
\[
(-\Delta+1)(2\pi \Ktwo(\cdot) + \log|\cdot|)  = \log|\cdot| \in W^{1,.p}_{loc}(\R^2)\mbox{ for all }p\in [1,2).
\]
Elliptic regularity thus implies that
 \begin{equation}
2\pi \Ktwo(\cdot) + \log|\cdot| \in W^{3,p}_{loc}(\R^2)\subset C^{1, \alpha }_{loc}(\R^2)
\qquad\mbox{ for all }p\in [1,2), \mbox{ with $\alpha = \frac{2p-2}p\in [0,1).$}
\label{K2log}\end{equation}
In particular, $L:= \lim_{z\to 0} (2\pi \Ktwo(z) +\log|z|)$ exists, and as a result,
\begin{equation}
s(x) =
S(x,x) 
=
\lim_{y\to x} \left[ \tilde S(x,y) + 2\pi \Ktwo(x-y) + \log|x-y| \right]  = \tilde S(x,x) + L.
\label{s.tildeS}\end{equation}
It follows from the definitions of $G$ (see \eqref{G.def}) and $\tilde S$ that for every $y\in\Omega$,
\begin{equation}
\begin{cases}
(-\Delta_x  + 1)\tilde S = 0 \qquad&\mbox{ for }x\in \Omega,  \\
\tilde S(x,y) =-2\pi \Ktwo(y-x) &\mbox{ for }x\in \partial \Omega .
\end{cases}
\label{tildeS.pde}\end{equation}
Then, since $\Ktwo\ge 0$, the maximum principle and \eqref{K2log} easily imply that 
for every $y$,
\begin{equation}\label{tildeS.bounds}
0 \ge \max_{x\in \bar\Omega} \tilde S(x,y)
\ge \min_{x\in \bar\Omega} \tilde S(x,y)
= \min_{x\in \partial \Omega} - 2\pi \Ktwo(x-y)
\ge \log (d(y)) -C(\Omega)
\end{equation}
for all $y\in \Omega$. This and \eqref{s.tildeS}  imply the estimate of $|s(x)|$
stated in \eqref{s.ests}.

\medskip

{\em Estimate of $|\nabla s|$}.
Next, we use the chain rule,  \eqref{s.tildeS}, and \eqref{tildeS.pde}  to compute
\[
\nabla s(x) = \left. \left( \nabla_x \tilde S(x,y) + \nabla_y\tilde S(x,y)\right)\right|_{y=x} \ 
=\left.  2\nabla_y \tilde S(x,y) \right|_{y=x}\ ,
\]
where the second equality follows from the standard fact that $\tilde S(x,y) = \tilde S(y,x)$ for
all $x$ and $y$.
By differentiating \eqref{tildeS.pde}, we find that
\[
\begin{cases}
(-\Delta_x  + 1)\partial_{y_j}\tilde S = 0 \qquad&\mbox{ for }x\in \Omega,  \\
\partial_{y_j} \tilde S(x,y) =-2\pi \partial_{y_j} \Ktwo(y-x) &\mbox{ for }x\in \partial \Omega ,
\end{cases}\quad\mbox{ for }j=1,2.
\]
From \eqref{K2log}, we see that $|\nabla \Ktwo(y-x)| \le Cd(y)^{-1}$ for $y\in \Omega$ and $x\in \partial \Omega$,
so we again use the maximum principle to deduce that
\begin{equation}\label{nytS}
\sup_x |\nabla_y\tilde S(x,y)| \le Cd(y)^{-1},\quad\mbox{ and thus }|\nabla s(x)|\le C d(x)^{-1}.
\end{equation}

\medskip

{\em Estimate of $|\nabla^2 s|$}.
Finally, we compute 
\begin{align}
\Delta s(x) 
&= \left. \left( \Delta_x \tilde S(x,y) + 2\nabla_x \cdot \nabla_y\tilde S(x,y) +\Delta_y\tilde S(x,y)\right)\right|_{y=x} \nonumber \\
&= - 2\tilde S(x,x) + \left. 2\nabla_x \cdot \nabla_y\tilde S(x,y) \right|_{y=x}\ . 
\label{Delta.s}
\end{align}
In general, if $w(x)$ satisfies $(-\Delta+1)w = 0$ in a ball $B(r,a)$, then standard elliptic theory\footnote{After rescaling, this is equivalent to the claim that if $(-\Delta+r^2)v=0$
on $B(1)$, then $|Dv(0)|\le C\|v\|_\infty$. This is proved by noting that in this context, 
standard interior estimates $\| v\|_{H^k(B(2^{-k}))} \le C(k) \|v\|_{L^2(B(1))}$ hold with constants 
independent of $r$.}
implies that 
\[
|\nabla w(a)| \le \frac C r \sup_{B(r,a)} |w|, 
\]
Fixing $y\in \Omega$, we apply this to  $w(x) = \partial_{y_j}S(x,y)$
in $B( d(y), y)$ and use \eqref{nytS} to conclude that
\[
|\nabla_x \partial_{y_j} \tilde S(y,y)| \le C d(y)^{-2}
\]
for every $y\in \Omega$. Then \eqref{Delta.s} and our earlier estimate of $|s|$ imply that $|\Delta s(x)|\le C d(x)^{-2}$.
\end{proof}

\subsection{An obstacle problem}

Having modified $H_\ep^N,$ we aim to study near minimizers in $M^*_{\ep, N}$ by characterizing an associated coincidence set.

The main result of this section is the following lemma, which yields some auxiliary 
functions that will play a key role in the proof of Proposition \ref{main.prop}. We introduce a family of obstacle problems indexed by a parameter $\lambda>0.$ 
 For each $\lambda$
in a certain range, this obstacle problem yields, among other things, a  {\it coincidence set}
$\Sigma_\lambda$, see \eqref{obs.output}.
We will see in Lemma \ref{L.uMepNstar} below that, given suitable $\hex$ and $N ,$
most vortices are found in  $\Sigma_\lambda$ for
the particular choice $\lambda =  \frac{2\pi \hex}N$.

This is very convenient, since the obstacle problem formulation allows for the use of barriers, not only to estimate how far the coincidence set is from $\partial\Omega$ but also, as we will see later, the minimum cost of a vortex lying outside the coincidence set in terms of its distance to it.

\begin{lemma}\label{L.obs}
Let $\xi_\ep = \xi_0+v_\ep$, where $v_\ep$ is the function
found in Lemma \ref{L.vep}. Then for $\hex\ge K_1$ and
$\lambda> ( |\Omega|- \hex^{-1/4})^{-1}$, there exist $m(\lambda)>0$, a function
$\vp_\lambda\in C^{1,1}_0(\Omega)$, and a set $\Sigma_\lambda\subset\Omega$
such that
\begin{equation}\label{obs.output}
\begin{aligned}
\zeta_\lambda
&:= \lambda\xi_\ep + \vp_\lambda \ge -m(\lambda), \\
\Sigma_\lambda 
&= \supp((-\Delta+1)\vp_\lambda) = \{ x\in \Omega : \zeta_\lambda(x) = -m(\lambda) \}
\end{aligned}
\end{equation}
and
\begin{equation}\label{probmeas}
\int_\Omega (-\Delta+1)\vp_\lambda \, dx = 1.
\end{equation}
Moreover, there exist positive constants $\lambda_0>\frac 1{|\Omega|}$ and
$c_2, \ldots ,c_5$ such that for all $\ep<\ep_0$,

\begin{itemize}
\item if $(|\Omega|- \hex^{-1/4})^{-1}\le \lambda \le\lambda_0$, then
\begin{align}
&\min_\Omega \zeta_\lambda = -m(\lambda)  \le   - c_2(|\Omega|-\frac 1\lambda)^2 , \label{obstacle2}\\
\zeta_\lambda(x) &  \ge  -2\sqrt{c_3\lambda} \left( |\Omega| - \frac 1{\lambda} \right) d(x) + \lambda d^2(x)
 \qquad\mbox{ if }d(x)< \sqrt{\frac{c_3}\lambda } \left(  |\Omega| - \frac 1{\lambda} \right)  \ .\label{obstacle3}
\end{align}

\item if $\lambda \ge \lambda_0$ then
\begin{equation}
\min_\Omega \zeta_\lambda = -m(\lambda)  \le  
 - c_5\lambda , \qquad \mbox{ and }\quad 
\zeta_\lambda(x)   \ge - c_4 d(x)\label{obstacle2a}.
\end{equation}
\end{itemize}
\end{lemma}

The new part of the lemma, apart from the  dependence on $\ep$ (which is very mild and
hence mostly suppressed in our notation),
consists of conclusions \eqref{obstacle2}, \eqref{obstacle3}, which of course
imply lower bounds on the distance between $\Sigma_\lambda$ and $\partial\Omega$.

Following ideas from \cite{SandSerfCMP}, Appendix A
we obtain $\vp_\lambda$ from an obstacle problem, stated as follows.
For $\lambda, m\ge 0$, define
\begin{align*}
\calO_{\lambda, m} &:=  \{ \vp \in H^1_0(\Omega) : \vp \ge - \lambda \xi_\ep - m  \} , \\
I(\vp) &:= \frac 12\int_{\Omega} |\nabla\vp|^2+\vp^2 \, dx . 
\end{align*}
Since $I(\cdot)$ is strictly convex and $\calO_\lam$ is convex and nonempty, we may define
\begin{align*}
\vp_\lam& := \mbox{ the unique minimizer of $I(\cdot)$ in $\calO_\lam$} , \\
\zeta_\lam &:=  \lambda \xi_\ep +\vp_\lam ,  \\
\Sigma_{\lambda, m} &:= \{ x\in \Omega :  \zeta_\lam(x) = - m \}.
\end{align*}
Below we will define $\vp_\lambda = \vp_\lam$ for a suitable choice $m=m(\lambda)$.

%
Well-known results about the obstacle problem (see for example \cite{KS}), 
guarantee that $\vp_\lam$ is $C^{1,1}$ and that
\begin{equation}\label{vplam1}
(-\Delta+1)\vp_\lam = \mathbf 1_{\Sigma_\lam}\cdot (\lambda w_\ep - m) \ge 0, \qquad
\mbox{ for }
w_\ep := (\Delta-1)\xi_\ep := 1 + o(\hex^{-1/4}).
\ee
It follows that  $\zeta_\lam\in C^{1,1}(\Omega)$, and
\begin{equation}\label{zetalam1}
(-\Delta+1)\zeta_\lam = - \lambda  w_\ep + \mathbf 1_{\Sigma_\lam}\cdot (\lambda w_\ep - m)  \ge -\lambda w_\ep.
\end{equation}
This equation allows us to control certain aspects of $\zeta_\lam$ by constructing sub- and supersolutions.

\begin{lemma}\label{L.barrier}
Assume that $\eta\in C^{1,1}(\Omega)$ and that $\eta\ge -m$ everywhere in $\Omega$.

\medskip

{\em Upper barrier}: If $\eta \ge0$ on $\partial \Omega$, 
and $(-\Delta+1)\eta \ge -\lambda w_\ep$  {\em a.e.} in $\Omega,$
then $\eta \ge \zeta_\lam$ in $\Omega$, and thus 
$\{ \eta = -m\} \subset \Sigma_\lam $.

\medskip

{\em Lower barrier}: If $\eta \le0$ on $\partial \Omega$, 
and $(-\Delta+1)\eta \le -\lambda w_\ep +  \mathbf 1_{ \{ \eta = -m \} }\cdot (\lambda w_\ep - m)$  {\em a.e.} in $\Omega,$
then $\eta \le \zeta_\lam$ in $\Omega$, and thus
$\Sigma_\lam \subset \{ \eta = -m\} $.
\end{lemma}

\begin{proof}
For an upper barrier $\eta$, it is clear that 
$ \zeta_\lam-\eta  =  - m-\eta \le 0$ on $\Sigma_\lam$, 
so the claim follows by applying the 
strong maximum principle (which in $2$ dimensions
requires only $W^{2,p}$ regularity for $p\ge 2$, see for example \cite{GT},  Theorem 9.1) to  $\zeta_\lam-\eta \le 0$  in $\Omega\setminus \Sigma_\lam$.
The other case similarly follows by applying the strong maximum principle to
to $\eta-\zeta_\lam$ on $\{ x\in \Omega : \eta(x)>-m\}$.
\end{proof}


\begin{proof}[Proof of Lemma \ref{L.obs}]
{\em Step 1}. Let
\[
f(\lambda, m) := \int_\Omega (-\Delta+1)\vp_\lam dx \ \overset{\eqref{vplam1}}= 
\int_{\Sigma_\lam}( \lambda w_\ep - m)\, dx.
\]
Lemma \ref{L.barrier}  and \eqref{zetalam1} immediately imply that
\[
\mbox{ if }0< m_1 < m_2, \quad\mbox{ then }
\frac{m_1}{m_2}\zeta_{\lambda, m_2} \ge  \zeta_{\lambda, m_1} \mbox{ and hence }\Sigma_{\lambda ,m_2}\subset \Sigma_{
\lambda ,m_1}.
\]
For $m_1<m_2$, since  $(\lambda w_\ep - m)\ge 0$ on $\Sigma_\lam$, 
it follows that
\begin{multline*}
f(\lambda, m_2)
=\int_{\Sigma_{\lambda, m_2}}( \lambda w_\ep - m_2)\, dx
\le
\int_{\Sigma_{\lambda, m_2}}( \lambda w_\ep - m_1)\, dx 
\le
\int_{\Sigma_{\lambda, m_1}}( \lambda w_\ep - m_1)\, dx
=
f(\lambda, m_1),
\end{multline*}
with strict inequality if $f(\lambda, m_2)>0$.

When $m=0$, it is clear that $\eta = 0$ is both an upper and lower barrier, and hence
that $\zeta_{\lambda,0} = 0$ and $\Sigma_{\lambda,0} = \Omega$.
Thus
\[
f(\lambda,0) \ = \lambda\int_\Omega w_\ep\, dx = \lambda |\Omega|( 1+
o(\hex^{-1/4})). 
\]
In particular, $f(\lambda, 0)>1$ if $\frac 1\lambda < |\Omega|- \hex^{-1/4}$
and $\hex$ is large enough. 

Similarly, if $m > \lambda\|\xi_\ep\|_\infty$, then $\eta = \lambda\xi_\ep$
is both an upper and lower barrier. Hence  $\zeta_\lam =  \lambda\xi_\ep$
(that is, $\vp_\lam = 0$), so $\Sigma_\lam = \emptyset$ and
\[
f(\lambda,m) = 0. 
\]
Since $f(\lambda, \cdot)$ is strictly decreasing
on its support, and compactly supported, there must then exist a unique
$m(\lambda)$ such that $f(\lambda, m(\lambda)) = 1$.

We now define $\vp_\lambda = \vp_{\lambda, m(\lambda)}$, and similarly
we set $\zeta_\lam = \zeta_{\lambda, m(\lambda)}$ and so on.
We have just shown that \eqref{probmeas} holds. The
regularity of $\vp_\lambda$ and properties
\eqref{obs.output}  follow directly from the construction of $\zeta_\lam$
and facts such as \eqref{zetalam1} about the obstacle problem.

{\em Step 2.} It remains to prove \eqref{obstacle2}, \eqref{obstacle3}. 
In doing so, we will first assume that
\begin{equation}\label{lambda.range}
(|\Omega|-\hex^{-1/4})^ {-1} \le\lambda \le \lambda_0
\end{equation}
for $\lambda_0 \in  (\frac 1{|\Omega|} ,\frac 2{|\Omega|})$ to be chosen below, independent
of $\ep \in (0,\ep_0]$.

For $r>0$ we will write 
\[
\calN_r := \{ x\in \Omega : d(x) < r\}.
\]
Fix $d_0$, depending on $\Omega$, such that the boundary distance function $d$ 
is smooth on $\calN_{2d_0}$,
and for $0<\delta \le d_0$ define
\[
\eta_{\delta, m} := f_{\delta, m}\circ d, \qquad\mbox{ for }
f_{\delta, m}(s) =  \begin{cases}
-\frac{2m s}\delta +  \frac {m s^2} {\delta ^2}  &\mbox{ if }s\le \delta\\
-m &\mbox{ if }s> \delta \, .
\end{cases}
\]
Then $\eta_{\delta, m}$ is $C^{1,1}$, and
\[
(-\Delta+1)\eta_{\delta, m} = 
\begin{cases}
-(\frac {2m}{\delta^2} + f_{\delta, m}'(d(x)) \, \Delta d(x) )  + f_{\delta, m}(d(x))
  &\mbox{ in }\calN_\delta\\
-m &\mbox{ in }\Omega\setminus {\overline\calN}_\delta.
\end{cases}
\]
We claim that there exist $m_0$ and $K_1$, depending only on 
$\Omega$, such that  if $0<m<m_0$ and $\hex\ge K_1$, then
\begin{equation}\label{etamdelta}
\eta_{\delta, m} \mbox{ is a }\begin{cases}
\mbox{ upper barrier  if }\delta = 2\sqrt{m/\lambda} \le d_0 \\
\mbox{ lower barrier  if }\delta = \sqrt{m/\lambda} \le d_0\ .
\end{cases}
\end{equation}
For example, if  $\delta  = \sqrt{m/\lambda}$, clearly $|f_{\delta, m}'| \le \frac {2m}\delta  =  2\sqrt{m\lambda}$ and $|f_{m,\delta}|\le m$, so (as long as $\delta\le d_0$)
\[
 (-\Delta+1)\eta_{\delta,m}   \le -2\lambda
 +  2\sqrt{m\lambda} \kappa +m \ \ \mbox{ in }\calN_\delta, \quad\mbox{ for }
\kappa := \|\Delta d\|_{L^\infty(\calN_{d_0})}.
\]
We require $K_1$ be large enough  that $w_\ep  < 3/2$ whenever $\hex\ge K_1$.
Then (recalling that $\lambda > |\Omega|^{-1}$) on sees that $\eta_{\delta, m}$ is a lower
barrier if $m_0$ satisfies
\[
m_0 \le\frac 14  |\Omega|^{-1}\ , \qquad 
2\sqrt{m_0}\kappa \le\frac 14  |\Omega|^{-1/2} , \qquad 
m_0 \le |\Omega|^{-1}d_0^2 \ ,
\]
where the first two conditions guarantee that $2\sqrt{m\lambda}\kappa+m \le \frac 12 \lambda$,
and the last condition guarantees that $\delta<d_0$. The case of an upper barrier
is essentially identical. Thus we have proved \eqref{etamdelta}.

\medskip

Next, we assume that $\hex\ge K_1$ and $0<m<m_0$, and we estimate $f(\lambda,m)$.
It follows from \eqref{etamdelta} that 
\[
\Omega \setminus \calN_{2\sqrt{m/\lambda}} \ \ 
\subset \ \  \Sigma_\lam \ \  \subset \ \ 
\Omega \setminus \calN_{\sqrt{m/\lambda}}
\]
Recalling \eqref{lambda.range}, we infer that there exist
$c < C$, depending only on $\Omega$, such that 
\[
|\Omega|-C\sqrt m \le |\Sigma_\lam|  \le  |\Omega| - c\sqrt m.
\]
Since $\lambda w_\ep - m = \lambda - m + o(\hex^{-1/4})$, it follows that
\[
f(\lambda, m) = \int_{\Sigma_\lam}(\lambda w_\ep-m) \,dx \le (\lambda - m)
(|\Omega| -  c \sqrt{m}) +o(\hex^{-1/4}).
\]
Thus if $c\sqrt m  =  2(|\Omega| - \lambda^{-1}) \overset{\eqref{lambda.range}}\ge 2 \hex^{-1/4} $, then
\[
f(\lambda, m) \le (\lambda - m) (\lambda^{-1} - \hex^{-1/4}) +o(\hex^{-1/4}) < 1
\]
after increasing $K_1$ if necessary.
Similarly, if $C\sqrt m = \frac 12 ( |\Omega| -\lambda^{-1}) \ge \frac 12 \hex^{-1/4}$, then
\[
f(\lambda , m) \ge  (\lambda - m )(\lambda^{-1} +C\sqrt m) +o(\hex^{-1/4}) >1
\]
after adjusting $K_1$.
Since $f$ is a decreasing function of $m$, we
 conclude that when $\hex \ge K_1$ and $0<m<m_0$, if $f(\lambda, m) = 1$, then there exist
constants $C>c>0$ such that
\[
c(|\Omega|-\lambda^{-1})^2 \le m \le C(|\Omega|-\lambda^{-1})^2 \, .
\]
In other words, 
\begin{equation}\label{mlambda}
c(|\Omega|-\lambda^{-1})^2 \le m(\lambda) \le C(|\Omega|-\lambda^{-1})^2.
\end{equation}
In particular \eqref{obstacle2} holds if \eqref{lambda.range} is satisfied.
(The condition $m<m_0(\Omega)$ translates to
the upper bound $\lambda\le \lambda_0$ in \eqref{lambda.range}.)

It now follows from \eqref{etamdelta} that 
$\eta_{\delta,m}$, with $m = m(\lambda)$ and $\delta = \sqrt{m/\lambda}$,
is a lower barrier for $\zeta_\lambda = \zeta_{\lambda, m(\lambda)}$. Since $m\mapsto \eta_{\delta, m}$
is a decreasing function, we conclude from this  and \eqref{mlambda} that
$\zeta_\lambda \ge \eta_{\delta, m}$ when $m = C(|\Omega|-\lambda^{-1})^2$
and $\delta = \sqrt{m/\lambda}$, as long as $d(x) < \frac c{\sqrt \lambda}(|\Omega|-\lambda^{-1})$. This is exactly conclusion \eqref{obstacle3}.

{\em Step 3}. 
Finally, we prove \eqref{obstacle2}, \eqref{obstacle3} for
$\lambda> \lambda_0 $. 

First for any $\delta\in (0,d_0)$ and $m>0$, we compute as above that
\[
(-\Delta+1) \eta_{\delta, m} \ge - \frac{2m}{\delta^2} - \frac {2 m}{\delta}\kappa - m
= - \left( \frac 2{\delta^2} +\frac{2\kappa}{\delta} +1\right)m.
\]
It follows that for any $\delta$ as above, there exists $\theta(\delta)>0$ such that
\[
(-\Delta+1) \eta_{\delta, \theta\lambda} \ge - \frac 12 \lambda \qquad\mbox{ whenever }0<\theta<\theta(\delta)
\]
and hence that $\eta_{\delta, \theta\lambda}$ is an upper barrier for
$\zeta_{\lambda, \theta\lambda}$. Here we are using the
assumption that $w_\ep \ge \frac 12$ everywhere, which we have
already imposed as a condition on $K_1$.
Then arguing as above, we estimate
\[
f(\lambda, \theta\lambda) \ge (|\Omega| - |\calN_{\delta}|)( 1 - \theta + o(\hex^{-1/4}) )\lambda .
\]
Since $\lambda_0 >|\Omega|^{-1}$, we may 
fix $\delta_0$ and $\theta_0<\theta(\delta_0)$  so small that 
$(|\Omega| - |\calN_{\delta_0}|)( (1 - \theta_0 )\lambda_0   >1$, and hence also 
$f(\lambda, \theta_0\lambda) >1$ when $\lambda\ge\lambda_0$.

The monotonicity of $f(\lambda, \cdot) $
then implies that $m(\lambda)\ge \theta_0\lambda$ whenever $\lambda \ge\lambda_0$.

It is also clear that $\vp_\lam\ge 0$ for all choices of parameters, and
hence that $\zeta_\lambda \ge  \lambda\xi_\ep$. Hopf's lemma implies that 
there exists a positive constant $c$ such that $\xi_0(x) \le -c d(x)$ for all
$x$, and it follows that $\xi_\ep(x) \le -(c/2)d(x)$ for all sufficiently large
$\hex$. This proves \eqref{obstacle2a}.

\end{proof}



\subsection{ Proof of Proposition \ref{main.prop} }

We now complete the proof of the proposition
by studying  near-minimizers of $\uH_\ep^N$, defined in \eqref{uH.def}. Here, we turn the delicate problem of estimating deviations of the energy due to small variations in the position of a single vortex. We reduce this problem, which is clearly nonlocal, to an almost local one by a screening procedure. This is a quantitative version of an argument\footnote{attributed by the authors of \cite{RNS} to unpublished work of Lieb} for a discrete energy similar to ours 
but simpler in some respects,
to bound from below minimum neighbor distances in minimizers. 


\begin{lemma}
There exists $c_0, c_1, \tep>0$ such that 
if $N, \hex$ satisfy \eqref{scaling1}, \eqref{scaling2}
and 
\[
a\in \uM^*_{\ep, N}:= \{ a = (a_1,\ldots, a_N)\in \Omega^N : \uH_\ep^N(a) \le \inf_{\Omega^N} \uH_\ep^N + \tep\},
\]
then\begin{align}\label{contain0}
\dist(a_i,\partial \Omega) &\ge  c_0  \hex^{-1/4} \mbox{ for all i}.
\\
\label{sep0}|a_i - a_j| &\ge  c_1\hex^{-1/2} \mbox{ for all }i\ne j 
\end{align}
\label{L.uMepNstar}\end{lemma}

Proposition \ref{main.prop} is an immediate corollary. Indeed,
It follows from \eqref{vep2} that $H_\ep^N = \uH_\ep^N$
in $M_{\ep,N}$ (defined in \eqref{Mepn.def}).
Thus for $a\in M_{\ep,N}^*\subset M_{\ep, N}$, 
\[
\uH_\ep^N(a) = H_\ep^N(a) 
\overset{\eqref{Mepnstar.def}}\le
 \inf_{M_{\ep,N}} H_\ep^N+\tep
=
 \inf_{M_{\ep,N}} \uH_\ep^N+\tep
\overset{\eqref{contain0}}=
\inf_{\Omega^N} \uH_\ep^N+\tep .
\]
Hence $a\in \uM_{\ep,.N}^*$, so Lemma \ref{L.uMepNstar} implies that $a$
satisfies \eqref{contain0}, \eqref{sep0}, proving Proposition \ref{main.prop}.

\begin{proof}[Proof of Lemma \ref{L.uMepNstar}]
Assume that $a = (a_1,\ldots, a_N)\in \uM_{\ep,N}^*$.

For the proof, we will write 
\[
\mu_\lambda := (-\Delta+1)\vp_\lambda\, dx   = \mathbf 1_{\Sigma_\lambda}(\lambda w_\ep
- m(\lambda)) \, dx
\]
that is, the measure on $\Omega$ whose density with respect to Lebesgue measure is
$(-\Delta+1)\vp_\lambda$. Recall also that $\mu_\lambda$ is a probability measure,
by \eqref{probmeas}.

\bigskip

{\em Proof of \eqref{contain0}}.
Let
\[
\lambda = \frac\hex {2\pi N} \overset{\eqref{scaling2}}\ge \left( |\Omega| - \hex^{-1/4} \right) ^{-1}.
\]
Then for any $\tilde a_1\in \Omega$, we
use the functions $\vp_\lambda$ and  $\zeta_\lambda = \lambda\xi_\ep +\vp_\lambda$ from Lemma \ref{L.obs}
to write
\begin{align*}
&\frac 1{4\pi^2N} \left[ \uH_\ep^N(a_1,a_2,\ldots, a_n) -
\uH_\ep^N(\tilde a_1,a_2,\ldots, a_n) \right] \\
&\qquad\qquad\qquad\qquad =
\lambda[\xi_\ep (a_1) - \xi_\ep(\tilde a_1)] +  \frac 1N \sum_{j=2}^N\big( G(a_1,a_j) - G(\tilde a_1, a_j)\big)  \\
&\qquad\qquad\qquad\qquad =   \left[  \zeta_\lambda(a_1) +U(a_1)\right] - 
 \left[ \zeta_\lambda(\tilde a_1) +U(\tilde a_1)\right]
 \end{align*}
for 
\[
U(x) = -\vp_\lambda(x) +\frac 1N \sum_{j=2}^N G(x,a_j), \quad\mbox{ so that }
\begin{cases}
(-\Delta+1)U = -\mu_\lambda +\frac 1N\sum_{j=2}^N \delta_{a_j}
&\mbox{ in }\Omega\\
U = 0 &\mbox{ on }\partial \Omega.
\end{cases}
\]
Since $a\in \uM_{\ep,N}^*$, it follows that 
\begin{equation}\label{contain1}
[\zeta_\lambda+U](a_1) \le \min_\Omega [\zeta_\lambda+U] + \frac{ \tep}{4\pi^2 N}.
\end{equation}

We claim that 
\begin{equation}
\inf_\Omega U <0.
\label{infU}\end{equation}
To prove this, assume toward a contradiction that $U\ge 0$ in $\Omega$. Then
\[
-\int_{\partial \Omega} \nu\cdot \nabla U
=
-\int_\Omega  \Delta U \\
\le
\int_\Omega(- \Delta +1) U \\
=
-1 + \frac {N-1}N <0,
\]
using the  equation for $(-\Delta+1)U$ and the fact that 
$\mu_\lambda$ is a probability measure. 
It follows that $\nu\cdot\nabla U(x) >0$ at some $x\in \partial\Omega$.
But since $U=0$ on $\partial \Omega$, this contradicts our assumption that 
$U\ge 0$ in $\Omega$, proving \eqref{infU}.

Now the boundary condition for $U$ implies that $\min_\Omega U<0$ is attained.
This cannot occur at any $a_j, j=2,\ldots, N$ (where $U=+\infty$)
or any any other point of  $\Omega\setminus \supp(\mu_\lambda)$, where $(-\Delta+1)U=0$.
Thus all minima of $U$ are contained in 
in $\Sigma_\lambda = \supp(\mu_\lambda)$, which is exactly the  set where 
$\zeta_\lambda$ attains its minimum. It follows that
\[
\min_\Omega [\zeta_\lambda+U]
=
\min_\Omega \zeta_\lambda+\min_\Omega U
\le 
\min_\Omega \zeta_\lambda+ U(a_1).
\]
Thus we infer from \eqref{contain1} that
\begin{equation}\label{contain2}
\zeta_\lambda(a_1) \le \min_\Omega \zeta_\lambda + \frac{\tep}{4\pi^2N}.
\end{equation}
Now \eqref{contain0} follows from Lemma \ref{L.obs}. To prove this
we consider two cases:

{\bf Case 1}:  $\frac 1{\lambda_0} \le \frac 1\lambda = \frac{2\pi N}\hex \le |\Omega| - \hex^{-1/4}$.

In this case, \eqref{obstacle2} implies that
\[
m(\lambda)\ge c_2 (|\Omega| - \frac 1 \lambda)^2 \ge c_2 \hex^{-1/2} \ge c_2(2\pi\lambda_0 N)^{-1/2} \ge 2 \frac \tep {4\pi^2 N}
\]
for any $N\ge 1$, if $\tep$ is small enough. 
As a result,
\[
\zeta_\lambda(a_i) \le \min \zeta_\lambda + \frac \tep {4\pi^2 N} = - m(\lambda)+ \frac \tep {4\pi^2 N} \le  - \frac 12 m(\lambda) \le -\frac 12 c_2 (|\Omega| -  \lambda^{-1})^2.
\]
On the other hand, it is clear from \eqref{obstacle3} that 
$\zeta_\lambda(x) \ge - 2\sqrt{c_3\lambda}(|\Omega|-\lambda^{-1})d(x)$ for
all $x$. In particular,
\[
d(a_i)  \ge - C\zeta_\lambda(a_i) (|\Omega|-\lambda^{-1})^{-1}  \ge c(|\Omega|-\lambda^{-1})
\ge c\hex^{-1/4},
\]
proving \eqref{contain0}.

\medskip

{\bf Case 2}:  $\frac 1{\lambda_0} \ge \frac 1\lambda = \frac{2\pi N}\hex $.

\smallskip

Then \eqref{contain0} follows by similar arguments, using 
\eqref{obstacle2a} in place of \eqref{obstacle2}, \eqref{obstacle3}.

\bigskip

{\em Proof of \eqref{sep0}}.
We now wish to prove a lower bound for the distance from $a_1$ to another point, say $a_2$. 
First, note that
\begin{equation}
U(a_1) \le \min_\Omega U + \frac{\tep}{4\pi^2N}  \overset{\eqref{infU}} < \frac{\tep}{4\pi^2N} 
\label{sep1}
\end{equation}
exactly as with \eqref{contain2}.
Next, we let
\[
r = c_6\hex^{-1/2}
\]
for $c_6$ to be fixed below, and we decompose
\[
U = 
 \left(-\vp_{near} + \frac 1 N G(\cdot, a_2) \right) +
 \left(-\vp_{far}  +   \frac 1N \sum_{j=3}^N G(x, a_j) \right)
=:U_{near} + U_{far}
\]
where $\vp_{near}$ and $\vp_{far}$ vanish on $\partial \Omega$ and satisfy
\[
\begin{aligned}
(-\Delta+1)\vp_{near}& = \mathbf 1_{B(r,a_2)}\cdot(\lambda w_\ep -m) , \\
(-\Delta+1)\vp_{far} &=  - \mathbf 1_{\Sigma_\lambda\setminus B(r,a_2)} \cdot(\lambda w_\ep -m)
\end{aligned}
\]
where $m= m(\lambda)$.
We will show below that if  $c_6$ is sufficiently small, then
\begin{equation}
0\le \max_{\partial B_{r}(a_2)} U_{near} 
< \min_{\bar B_{r/2(a_2)}} U_{near} - \frac{\tep}{4\pi^2 N}.
\label{Ungrow}\end{equation}
Accepting this for the moment, we
assume toward a contradiction that $|a_1 - a_2| \le \frac r 2$. 
Then
\begin{align*}
U_{far}(a_1) 
= U(a_1)-U_{near}(a_1)
&\overset{\eqref{sep1}}\le
(\min_{\partial B_r(a_2)}U + \frac{\tep}{4\pi^2 N} ) - U_{near}(a_1) \\
&\le
\min_{\partial B_r(a_2)}U_{far}
+
\max_{\partial B_r(a_2)}U_{near} + \frac{\tep}{4\pi^2 N} 
 - U_{near}(a_1) \\
&\overset{\eqref{Ungrow}}< 
\min_{\partial B_r(a_2)}U_{far} 
\end{align*}
and
\[
U_{far}(a_1) 
= U(a_1)-U_{near}(a_1)
< 0 \quad\mbox{ by }\eqref{sep1}, \eqref{Ungrow}.
\]
It follows that $U_{far}$ attains a negative local minimum in $B_r(a_2)$.
But this cannot happen, by the maximum principle, since
\[
(-\Delta +1)U_{far} =
 \mathbf 1_{\Sigma_\lambda\setminus B(r,a_2)} \cdot(\lambda w_\ep -m)
+ \sum_{j\ge 3, \ \  |a_j - a_2| < r} \delta_{a_j}  \ge 0\quad\mbox{ in }B_r(a_2).
\]
Thus $|a_1-a_2| >\frac r2$, proving \eqref{sep0}.

\medskip

{\em Proof of \eqref{Ungrow}}.
Let $r= c_6\hex^{-1/2} \le \frac 12$. 

Recall from \eqref{S.def}
that 
\[
G(x,y) 
= \frac 1{2\pi} \left( - \log|x-y| + S(x,y)\right) .
\]
In addition, it follows from  \eqref{splitG}, \eqref{K2log}
that $S$ can be written as the sum of a radial $C^{1,\alpha}$ function 
and a function $\tilde S(x,y)$ that satisfies certain estimates recorded in
\eqref{tildeS.bounds}, \eqref{nytS}.
These imply that 
\begin{equation}\label{SOS}
|S(x,y)| \le C  -   \log d(y) , \qquad
|\nabla_x S(x,y)|\le C d(x)^{-1}.
\end{equation}
Since\begin{align*}
\vp_{near}(x) = \int_\Omega G(x,y)(-\Delta+1)\vp(y)\, dy =
 \int_{B(r,a_2)} G(x,y)(\lambda w_\ep(y)-m) \, dy,
\end{align*}
we can write $U_{near}$ in the form
\be 
U_{near}(x)
= - \frac 1{2\pi N}\log|x-a_2| +
\psi_1(x) +\psi_2(x)
\label{rewriteUnear}\ee
 where
 \begin{align*}
 \psi_1(x) &:= 
\frac 1 {2\pi} \int_{B(r,a_2)}\log|x-y| (\lambda w_\ep(y)-m) \, dy\,, \\
\psi_2(x) &:= \frac 1{2\pi N}S(x,a_2) - \left(  \frac 1 {2\pi}  \int_{B(r,a_2)}  S(x,y) (\lambda w_\ep(y)-m) \, dy 
 \right).
\end{align*}
We have arranged that $\hex$ is large enough that 
$w_\ep \le 2$.
Noting that $\log|x-y|<0$
for $x,y\in B_r(a_2)$,  we therefore have
\begin{equation}\label{psi1a}
 \left|\psi_1(x) \right|
 \le
\left|\frac \lambda {\pi} \int_{B(r,a_2)}\log|x - y| \, dy\right| 
\le \lambda r^2 (|\log r| + \frac 12)
\le \frac {c_6^2}{\pi N} (\log \hex +C) \ \ \ \mbox{ for }x\in B(r,a_2)
\end{equation}
where we have used the classical fact that the  integral  is maximized when $x=a_2$,
together with the choices of $\lambda$ and $r$.
Similarly
\begin{equation}\label{psi1b}
|\nabla\psi_1(x)| \le2\lambda r,\qquad\mbox{ and thus }|\psi_1(x)-\psi_1(y)|\le 4\lambda r^2 = \frac {2c_6^2}{N\pi}\ \ \ 
\mbox{ for }x,y\in B_r(a_2).
\end{equation}

Also,  it follows easily from \eqref{SOS} and \eqref{contain0} that for $x\in B_r(a_2)$, 
\begin{equation}\label{psi2a}
|\psi_2(x)| \le  \left( \frac 1{2\pi N} + \frac{c_6^2}{2N} \right)(1+\log \hex^{1/4})
\end{equation}
and
\begin{equation}\label{psi2b}
|\nabla\psi_2(x)| \le \frac C N (1+c_6^2)  \hex^{1/4}\qquad\mbox{ and thus }|\psi_2(x)-\psi_2(y)|\le \frac CN \hex^{-1/4}
\mbox{ for }x,y\in B_r(a_2).
\end{equation}

It follows from \eqref{rewriteUnear}, \eqref{psi1a}, \eqref{psi2a} that if $|x-a_2|= r = c_6\hex^{-1/2}$, then
\[
U_{near}(x) \ge \frac 1{2\pi N}( \log \hex^{1/2} - \log \hex^{1/4})  - 
 \frac {c_6^2} N (1+\log\hex)
 = \frac 1{8\pi N}\log \hex -  \frac {c_6^2} N (1+\log\hex).
\]
This is clearly positive if $c_6$ is small enough, whence the first inequality in \eqref{Ungrow}. Similarly, if $|x-a_2| = r$ and $|y-a_2|\le r/2$, then we deduce from 
\eqref{rewriteUnear}, \eqref{psi1b}, \eqref{psi2b} that 
\[
U_{near}(y) - U_{near}(x)  \ge \frac 1{2\pi N}\log 2  - \frac CN(c_6^2 + o(\hex^{-1/4})).
\]
By choosing $K_1$ large, $c_6$ small and $\tep$ small, we can therefore clearly arrange
that 
\[
U_{near}(y) - U_{near}(x) > \frac \tep{4\pi^2N},
\]
completing the proof of \eqref{Ungrow}.
\end{proof}

\section{Lower bounds and localization} \label{sec:3}

In this section, we prove some results about pairs $(u,A)$ for which there exists
$N\in \N$ and $a = (a_1,\ldots, a_N)\in \Omega^N$ such that
\begin{equation}\label{close}
\left\| J(u) - \sum_{j=1}^N \pi \delta_{a_j} \right\|_{\dot{W}^{-1,1}(\Omega)} \leq \sigma_\ep 
\end{equation}
and
$\sigma_\ep \ll \rho_a = \frac 14 \min\{ \min_{i\ne j}|a_i-a_j|, \min_i\dist(a_i,\partial \Omega) \}$.
These are all adapted from results in \cite{JSp2}  about the simplified Ginzburg-Landau functional without magnetic field.

Our first result provides very precise lower bounds for $GL_\ep(u,A)$ when
\eqref{close} holds and $\sigma_\ep$ is small enough, for suitable values
of other parameters such as $N, \rho_a$, etc.

\begin{prop}\label{P.gstab}
Let  $\Omega$ be a
bounded, open simply connected subset of $\R^2$ with $C^1$ boundary. Then there
exists absolute constants $C$ and $C_1$ with the following property:

Assume that $(u,A)\in H^1(\Omega; \C) \times H^1(\Omega; \R^2)$, 
and that the Coulomb gauge condition \eqref{coulomb} holds.
If there exist
 $a = ( a_1, \ldots, a_N ) \in \Omega^{N*}$, for some $N\ge 1$,
such that  \eqref{close} holds and
\begin{equation}  
4\ep\sqrt{\ln(\rho_a/\ep)} \le  \ 4 \sigma_\ep  \ 
 \le \  \sigma^*  \ :=  \left[ \frac{\rho_a}{N^{3}}(\sigma_\ep + \ep E_\ep(u))\right]^{1/2}
 \le \frac{\rho_a}{N C_1}
\label{gstab.h1}\end{equation}
then
\begin{equation}
GL_\ep(u,A) \ge
H_\ep^N(a) +\kappa_\ep^{GL}
- C\left[ \frac{N^{5}}{\rho_a}
( \sigma_\ep + \ep {E_\ep(u)} ) \right]^{1/2} -  \varpi , 
\label{gstab.c1}\end{equation}
where 
\begin{align}\label{varpi.def}
\varpi  = \varpi(u) &= C\left(\ep E_\ep(u)^2 + \ep \hex^4 + \sigma_\ep^{\frac 78}(E_\ep(u)+\hex^2)^{\frac 5 8}\right)\,, \ \ \mbox{ and }
\\
\label{kappaGL}
\kappa_\ep^{GL} &= 
 \hex^2 F(\xi_0) + N (\pi\ln \frac 1\ep + \gamma) .
\end{align}
Here $F(\cdot)$ was defined in \eqref{F.def}, and the definition of $\gamma$ 
 is discussed following \eqref{BBHsurplus} below.

\end{prop}

The next proposition 
shows that if  $(u,A)$ satisfies \eqref{close} and
nearly attains the energy lower bound in
\eqref{gstab.c1}, then in fact \eqref{close} can be strengthened
significantly, after possibly adjusting the points
$a\in \Omega^N$ slightly.

\begin{prop}\label{P.localization}
Let $\Omega$ be a bounded, open, simply connected subset of $\R^2$
with $C^1$ boundary. Then there exist constants $C$ and $C_2$,
depending on diam$(\Omega)$, with the following property:

Assume that $(u,A)\in H^1(\Omega; \C) \times H^1(\Omega; \R^2),$ 
and that the Coulomb gauge condition \eqref{coulomb} holds.
If there exist
 $a = ( a_1, \ldots, a_N ) \in \Omega^{N*}$, for some $N\ge 1$,
such that \eqref{close} holds with
\begin{equation}\label{loc.h1}
\sigma_\ep \le \frac {\rho_a}{8C_2N^5},
\end{equation}
and if in addition 
\begin{equation}
\mbox{$E_\ep(u) \ge 1$ and }
\quad \left[ \frac{N^5}{\rho_a}
E_\ep(u) + \frac {N^{10}}{\rho_a^2}\sqrt{E_\ep(u)}
\right] \le \frac 1 \ep,
\label{Ebds}\end{equation}
then there exist $(\xi_1,\ldots, \xi_N)\in \Omega^{N*}$ such that
$|\xi_i - a_i|\le  \frac{\rho_a}{2 C_2 N^4} $ for all $i$, and 
\begin{align}
&\left\| J(u) - \pi \sum _{i=1}^N \delta_{\xi_i} \right\|_{\dot W^{-1,1} }  
\label{p2.c1}
\\
&\hspace{5em}\le \
C\ \ep\left[N (C +\widetilde \Sigma_\ep^{GL})^2
e^{\frac 1\pi\widetilde\Sigma_\ep^{GL}} + 
(C +\widetilde \Sigma_\ep^{GL}) \frac{N^5}{\rho_a}+ {E_\ep(u)}
\right]\nonumber
\end{align} 
where
\be\label{SigmaGL}
\Sigma_\ep^{GL}: = \Sigma_\ep^{GL} (u,A,a)
= GL_\ep(u,A) - \kappa_\ep^{GL} - H_\ep^N(a),
\qquad 
\widetilde \Sigma_\ep^{GL}: = \Sigma_\ep^{GL}(u,A,a)+\varpi(u).
\ee
\end{prop}

Results very much like Proposition \ref{P.gstab} 
are proved in \cite{KSp}, Theorem 4.1,
but with the leading terms on the right-hand side of
\eqref{gstab.c1} appearing in a different form 
that is not well-suited to our purposes. 
There does not seem to be a counterpart of Proposition
\ref{P.localization} in \cite{KSp}.


\subsection{ A reduction}

Both propositions are proved by reducing them 
to results in \cite{JSp2}. These
relate 
the simplified Ginzburg-Landau energy $E_\ep$
to the {\em renormalized energy}  $W^N:\Omega^{N*}\to \R$,
introduced in the pioneering
book of Bethuel, Brezis and H\'elein \cite{BBH}, and
defined in our context by
\begin{equation}
W^N(a) = -\pi\sum_{i\neq j} \log |a_i - a_j|  + \pi\sum_{i , j} R(a_i, a_j).
\label{W.def0}\end{equation}
Here $R(\cdot, \cdot)$ is the regular part of the Green's function for
the Laplace operator on $\Omega$, {\em i.e.}
\[
R(x,y) = 2\pi \Gamma(x,y) + \log|x-y|,
\]
where $\Gamma  = \Gamma(x,y)$ is characterized by
\[
-\Delta_x\Gamma = \delta_y\mbox{ for }x\in\Omega,\qquad
\Gamma(x,y)= 0\mbox{ for }x\in \partial \Omega\, .
\]
The results we will use from \cite{JSp2} involve
the quantity
\begin{equation}\label{BBHsurplus}
\Sigma_\ep^{BBH} = \Sigma_\ep^{BBH} (u,a) = E_\ep(u) - \kappa_\ep^{BBH} - W^N(a), \qquad \mbox{ for } \ \kappa_\ep^{BBH} := N(\pi \log\frac 1 \ep +\gamma).
\end{equation}
where $\gamma$ is the same constant appearing in \eqref{kappaGL},
whose definition (which we will not actually need) can be
found in \cite{BBH}, Lemma IX.1, where the constant was first introduced.

The reduction of Propositions \ref{P.gstab} and \ref{P.localization}
to results from \cite{JSp2} will be carried out by proving,
roughly speaking,  that $\Sigma_\ep^{GL}(u,A,a)
\approx \Sigma_\ep^{BBH}(u,a)$ in the regimes we are interested in.
More precisely, we will prove

\begin{lemma}
Assume that $u\in H^1(\Omega;\C)$ satisfies
\eqref{close} and that $N \le  C \hex$.
\begin{align}
\left|\min_{A {\scriptsize \ satisfying \eqref{coulomb}}} \Sigma^{GL}_\ep(u,A,a)
- 
\Sigma^{BBH}_\ep(u,a) \right|
\le
\varpi(u) . 
\label{GLandE0}
\end{align}
\label{L.reduction}
\end{lemma}

If this is known, both propositions follow essentially by transcribing results from \cite{JSp2}:


\begin{proof}[Proof of Proposition \ref{P.gstab}.]
This is a direct consequence of  Lemma \ref{L.reduction} and  \cite{JSp2}, Theorem 2.
Indeed, the hypotheses of Proposition \ref{P.gstab} imply those of \cite{JSp2}, Theorem 2, and one of the conclusions of that result is the lower bound
\[
E_\ep(u) \ge
W^N(a) +\kappa_\ep^{BBH} +\mbox{ other positive terms}
- C\left[ \frac{N^{5}}{\rho_a} 
( \sigma_\ep + \ep {E_\ep(u)} ) \right]^{1/2} .
\]
This and \eqref{GLandE0} immediately imply \eqref{gstab.c1}.
(In fact they imply a stronger result, one that we have not recorded
here, as we do not need the extra positive terms. See \cite{JSp2} or 
\cite{KSp} for more.)
\end{proof}

\begin{proof}[Proof of Proposition \ref{P.localization}.]
This follows immediately from Lemma \ref{L.reduction} and  \cite{JSp2}, Theorem 3. Indeed, Proposition \ref{P.localization} and the cited result from \cite{JSp2}
have exactly the same hypotheses. The conclusion of the result in \cite{JSp2}
is \eqref{p2.c1},  but with $\Sigma_\ep^{BBH}(u,a)$ in place of 
$\widetilde \Sigma_\ep^{GL}(u,A,a)$. Since Lemma \ref{L.reduction}
implies that $\Sigma_\ep^{BBH}(u,a) \le  \widetilde \Sigma_\ep^{GL}(u,A,a)$,
the result follows.\end{proof}

\subsection{Proof of Lemma \ref{L.reduction}}

The proof has two main steps. The first is to show that
$ \Sigma_\ep^{BBH}(u,a)$ is 
close to
\begin{equation}\label{wtSigma}
\widetilde \Sigma_\ep^{GL}(u,a)
: = 
GL_\ep^{min}(u) - \min_{B\in H^2\cap H^1_0} \left( \Phi(B)  - 2\pi \sum_{i=1}^N B(a_i) \right)
- W^N(a)
-
\kappa_\ep^{BBH} ,
\end{equation}
where 
\be
GL_\ep^{min}(u) := 
\min_{A\mbox{ \scriptsize satisfying \eqref{coulomb}}} GL_\ep(u,A)
,\quad\mbox{ and }\quad\Phi(B )
:=  \frac 12 \int_\Omega  |\nabla^\perp B|^2 + (\Delta B  + \hex )^2.
\label{some.defs}
\ee
We will then check that 
\[
\min_{A\mbox{ \scriptsize satisfying \eqref{coulomb}}}
 \Sigma_\ep^{GL}(u,A,a)
=
\widetilde \Sigma_\ep^{GL}(u,a).
\]

\begin{lemma}\label{L.tildered}Assume that $u\in H^1(\Omega;\C)$ satisfies
\eqref{close} and that $N \le  \hex|\Omega|/2\pi$. Then
for any $\theta\in (0,1)$
\begin{align}\label{GLandE}
\left|\widetilde \Sigma^{GL}_\ep(u,a)
- 
\Sigma^{BBH}_\ep(u,a) \right|
\le \varpi\, . 
\end{align}
The implicit constants in \eqref{GLandE} depend on $\Omega, \theta$ and the assumed bound
on  $N/\hex$.
\end{lemma}


\begin{proof}
{\em Step 1. Preliminaries.}
We start from the algebraic identity
\begin{equation}
GL_\ep(u,A) 
\ = \ E_\ep(u)
\ - \ \int_{\Omega} j(u)\cdot A 
+\frac 12 \int_\Omega  |A|^2 + |\nabla\times A - \hex |^2
+ R(u,A)
\label{G.split}\end{equation}
where
\begin{equation}
R(u,A):= \frac 12 \int_\Omega (|u|^2-1)|A|^2.
\label{R.def}\end{equation}
We rewrite in terms of $B = (\nabla^\perp)^{-1}A$,
see \eqref{B.def},
then integrate by parts to find that
 \begin{equation}\label{GsplitB}
GL_\ep(u,\nabla^\perp B) 
\ = \ E_\ep(u)
\ - \ 2 \int_{\Omega}B  \, Ju
\ + \ 
\Phi(B ) + R(u, \nabla^\perp B).
\end{equation}
Thus
\[
GL_\ep(u,\nabla^\perp B) - \left[ E_\ep(u) + \left( \Phi(B) - 2\pi \sum_{i=1}^N B(a_i)
\right)\right]
= - R(u,\nabla^\perp B) -  2\int_\Omega B\left(Ju  - \pi \sum \delta_{a_i} \right).
\]

\medskip

{\em Step 2: lower bound for $\widetilde\Sigma_\ep^{GL}(u,a) $ }. 
We now prove that
\begin{equation}\label{SversusS1}
\widetilde\Sigma_\ep^{GL}(u,a)  - \Sigma_\ep^{BBH}(u,a) \gtrsim 
- \ep \left(E_\ep(u)^2 + \hex^4 \right) -   \sigma_\ep ^{1-\frac 2m}(E_\ep(u)+\hex^2)^{\frac 12 + \frac 2m}.
\end{equation}
For this, let $A_* = \nabla^\perp B_*$ minimize $B\mapsto GL_\ep(u,\nabla^\perp B)$,
so that $GL_\ep^{min}(u) = GL_\ep(u, \nabla^\perp B_*)$.
Then
\be\label{BH2}
E_\ep(u) + \frac 12|\Omega| \hex^2 = GL_\ep(u,0) \ge GL_\ep(u,\nabla^\perp B_*) \ge \int_\Omega |\Delta B_* +\hex|^2.
\ee
Since $B_*=0$ on $\partial \Omega$, basic elliptic estimates (see for example \cite{Evans}, Section 6.3.2, Theorem 4 and the remark that follows)
imply that
\begin{equation}\label{need.soon}
\|  B_* \|_{H^2} \le C \|\Delta B_* \|_{L^2} \le C (\| \Delta B_* + \hex\|_{L^2} + \hex)
\le C ( \sqrt{E_\ep(u) }+ \hex).
\ee
Since $W^{1,m}$ embeds into $H^2$ for every $m<\infty$ (and since we always assume $\hex\ge 1$), it follows that
\be\label{R.est}
|R(u, \nabla^\perp B_*)| \le   \int_\Omega \frac{(|u|^2-1)^2}{\ep} + \ep |\nabla^\perp B_*|^4\le 
C \ep \left( E_\ep(u)^2+  \hex^4\right).
\ee
For the remaining term, for any $m>2$ we estimate
\be\label{enough}
\left|\int B_* \ \left( Ju - \pi \sum_{i=1}^N \delta_{a_i} \right) \right|
\le
 \| Ju -  \pi \sum_{i=1}^N \delta_{a_i} \|_{\dot W^{-1,\frac m{m-1}}} \ \|  B_* \|_{\dot W^{1,m}_0}.
\ee
For any $m>2$, an interpolation inequality and \eqref{close} yield
\begin{align}
 \| Ju -  \pi \sum_{i=1}^N \delta_{a_i} \|_{\dot W^{-1,\frac m{m-1}}}
& \le
\| Ju -  \pi \sum_{i=1}^N \delta_{a_i} \|_{\dot W^{-1,1}}^{1 - \frac 2m}
\| Ju -  \pi \sum_{i=1}^N \delta_{a_i} \|_{L^1}^{\frac 2m}\nonumber \\
& \le
\sigma_{\ep}^{1-\frac 2m} \ 
\| Ju -  \pi \sum_{i=1}^N \delta_{a_i} \|_{L^1}^{\frac 2m}\nonumber\\
&\le
C \sigma_\ep^{1-\frac 2m} \left( E_\ep(u) + \pi N\right)^{\frac 2m},
\label{Jdiff}\end{align}
since 
\[
\| Ju -  \pi \sum_{i=1}^N \delta_{a_i} \|_{L^1} \le
\| Ju \|_{L^1} + \|  \pi \sum_{i=1}^N \delta_{a_i} \|_{L^1} \le C E_\ep(u) + \pi N.
\]
Since $N\le \hex|\Omega|/2\pi$,
it follows that 
\begin{equation}\label{laters}
\begin{aligned}
GL_\ep^{min}(u) = GL(u, \nabla^\perp B_*)& \ge E_\ep(u) + \left( \Phi(B_*)-2\pi \sum_{i=1}^N B_*(a_i)\right) \\
&\qquad - C\ep \left(E_\ep(u)^2 + \hex^4 \right) -  C \sigma_\ep ^{1-\frac 2m}(E_\ep(u)+\hex^2)^{\frac 12 + \frac 2m}.
\end{aligned}
\end{equation}
Subtracting $W^N(a)+\kappa_\ep^{BBH}$ from both sides and rearranging,
we obtain \eqref{SversusS1}

\medskip

{\em Step 3: upper bound for $\widetilde\Sigma_\ep^{GL}(u,a) $}.
The opposite inequality is proved in a very similar way.
First, for any $B$, 
\[
\| B\|_{H^2} \le C \|\Delta B\|_{L^2} \le C (\| \Delta B + \hex\|_{L^2} + \hex)
\le C ( \sqrt{\Phi(B)} + \hex).
\]
Thus, using  the bound $N\le \hex|\Omega|/2\pi$,
\begin{equation}\label{gdi}
\Phi(B)-2\pi \sum B(a_i) \ge \Phi(B) - NC  ( \sqrt{\Phi(B)} + \hex) . 
\ee
It follows from this and elementary inequalities, together with our standing assumption $N\le C|\Omega|$, that
\be\label{Phi-sumB}
\Phi(B)-2\pi \sum B(a_i) \ge  - C\hex^2 .
\ee
Thus standard lower semicontinuity arguments imply that the minimum
of the left-hand side is attained. Let $\beta$ denote a minimizer. 
It is clear that $\Phi(\beta)-2\pi \sum \beta(a_i)\le 0$, since
otherwise we could decrease the value of the functional by taking $\beta=0$.
Then 
it follows from \eqref{gdi} that 
\be\label{beta.est1}
\Phi(\beta) \le C \hex^2, \qquad \mbox{ and hence } \quad \| \beta\|_{H^2}^2 \le C\hex^2.
\ee
Thus $\beta$ satisfies the same estimate as $B_*$ in \eqref{need.soon} --- a slightly stronger estimate actually, although we will not use the  improvement. 
We can thus estimate the error terms exactly as above
to conclude that
\[
\begin{aligned}
GL_\ep^{min}(u) \le GL(u, \nabla^\perp \beta) &\le E_\ep(u) +  \Phi(\beta) - 2\pi \sum_{i=1}^N \beta(a_i)\\
&\qquad 
+ C\ep \left(E_\ep(u)^2 + \hex^4 \right)  +  C \sigma_\ep ^{1-\frac 2m}(E_\ep(u)+\hex^2)^{\frac 12 + \frac 2m}.
\end{aligned}
\]
Recalling the choice of $\beta$ and rewriting as above, this implies that 
\[
\widetilde\Sigma_\ep^{GL}(u,a)  - \Sigma_\ep^{BBH}(u,A) \lesssim 
- \ep \left(E_\ep(u)^2 + \hex^4 \right) -   \sigma_\ep ^{1-\frac 2m}(E_\ep(u)+\hex^2)^{\frac 12 + \frac 2m}.
\]
Choosing $m = \frac 1 {16}$ completes the proof.
\end{proof}

To finish the proof of Lemma \ref{L.reduction},
we will rewrite $\widetilde\Sigma_\ep^{GL}(u,a)$, which was defined in
\eqref{wtSigma}, \eqref{some.defs}. 
Toward this end, 
we fix $a= (a_1,\ldots, a_N)\in \Omega^N$, and we
let $\beta \in H^2\cap H^1_0(\Omega)$ denote the unique minimizer
of
\[
B\mapsto   \Phi(B ) -2 \pi \sum_{i=1}^N  B(a_i) .
\]
We want to find a simple expression for $\Phi(\beta) - 2\pi \sum_{i=1}^N \beta(a_i)$.

\begin{lemma}$\beta$ belongs to $W^{3,p}(\Omega)$ for every $p<2$, and
satisfies
\begin{align*}
\Delta^2 \beta -\Delta \beta &= 2 \pi \sum_{i=1}^N \delta_{a_i}\mbox{ in }\Omega,
 \qquad &\beta=0\mbox{ on }\partial \Omega,
 \qquad &-\Delta \beta=\hex\mbox{ on }\partial \Omega.
\end{align*}
\end{lemma}

We omit the proof, which is standard.

We now define  $B_1:\Omega\to \R$ as the solution of the boundary value problem
\begin{align*}
\Delta^2 B_1 -\Delta B_1 &= 2 \pi \sum_{i=1}^N \delta_{a_i}\mbox{ in }\Omega,
 \qquad &B_1=0\mbox{ on }\partial \Omega,
 \qquad &-\Delta B_1=0\mbox{ on }\partial \Omega.
\end{align*}
With this notation we can state

\begin{lemma}
\[
\min_{B\in H^1_0\cap H^2}\Big[ \Phi(B) -2\pi \sum_i B(a_i)\Big]
= \hex^2 F(\xi_0)  + 2\pi \hex \sum_i \xi_0(a_i) - \pi \sum_i B_1(a_i).  \ \ \ \
\]
where $\xi_0$ is defined in \eqref{xi0.def}.
\label{lem.rewrite1}\end{lemma}

\begin{proof}
By differentiating the equation satisfied by $\xi_0$, one finds that
\begin{align*}
\Delta^2 \xi_0 -\Delta \xi_0 &= 0 \mbox{ in }\Omega,
 \qquad &\xi_0=0\mbox{ on }\partial \Omega,
\quad  &-\Delta \xi_0=-1\mbox{ on }\partial \Omega.
\end{align*}
It follows that
\begin{equation}
\beta= -\hex\xi_0+B_1.
\label{rewrite1.1}\end{equation}
Defining $w_1 = -\Delta B_1$, 
it follows that $\Delta \beta = - \hex (\xi_0+1)- w_1$,  
and hence that 
\begin{equation}
(\hex +\Delta \beta)^2 = (\hex \xi_0 + w_1)^2.
\label{rewrite1.2}\end{equation}
Furthermore,
\begin{equation}
-\Delta(w_1+B_1) = (-\Delta+1) w_1 = (\Delta^2-\Delta)B_1 = 2\pi\sum_i \delta_{a_i}.
\label{rewrite1.3}\end{equation}
Using \eqref{rewrite1.1} and \eqref{rewrite1.2}, we  rewrite
\begin{align*}
\Phi(\beta)
&=
\frac 12\int |\nabla \beta|^2 + (\Delta \beta + \hex)^2\\
&=
\frac 12 \int  |\nabla( \hex \xi_0 - B_1)|^2 + (\hex \xi_0 + w_1)^2\\
&=
\hex^2\left(\frac 12 \int |\nabla \xi_0|^2 + \xi_0^2 \right)
+ \hex \left( 
\int -\nabla \xi_0 \cdot \nabla B_1  + \xi_0 w_1 \right)
+ \left(\frac 12 \int |\nabla B_1|^2 + w_1^2\right).
\end{align*} 

For the second term on the right-hand side, note that
\begin{align*}
\int \nabla \xi_0 \cdot \nabla B_1  - \xi_0 w_1 
&=0, \qquad
\mbox{ since $\xi_0\in H^1_0$ and $-\Delta B_1 = w_1$.}
\end{align*}
For the final term on the right-hand side, since $B_1=w_1=0$ on $\partial \Omega$ and $w_1=-\Delta B_1$, we integrate by parts and use the equations (in particular \eqref{rewrite1.3}) to find that
\begin{align*}
\frac 12 \int |\nabla B_1|^2 + w_1^2
&=
\frac 12 \int  - B_1 \Delta B_1 - w_1 \Delta B_1\\
&=
-\frac 12 \int  (B_1 + w_1)\Delta B_1\\
&=
\pi \sum B_1(a_i).
\end{align*}
The conclusion now follows by using these facts and \eqref{rewrite1.1}
to rewrite $\Phi(\beta) -2\pi \sum_i \beta(a_i)$.
\end{proof}

Finally, we complete the

\begin{proof}[Proof of Lemma \ref{L.reduction}]

In view of Lemma \ref{L.tildered}, we must show that
\[
\min_{A\mbox{ \scriptsize satisfying \eqref{coulomb}}}
 \Sigma_\ep^{GL}(u,A,a)
=
\widetilde \Sigma_\ep^{GL}(u,a).
\]
By comparing the definitions,
we see that this is the same as
\be\label{WHrelation}
H_\ep^N(a) + \hex^2 F(\xi_0) = 
\min_{B\in H^1_0\cap H^2}\Big[ \Phi(B) -2\pi \sum_i B(a_i)\Big]
+ W^N(a).
\ee
Using the previous lemma and various definitions,
see \eqref{S.def}, \eqref{Hep.newdef}, \eqref{W.def0}, 
this reduces to checking that
\[
\pi \sum_{i,j=1}^N S(a_i,a_j) = 
\pi \sum_{i,j=1}^N R(a_i,a_j)  - \pi \sum_{i=1}^N B(a_i)  .
\]
So we only need to prove that
\[
B_1(x)  \ = \  \sum_j  R(x,a_j) - S(x,a_j)
\]
for all $x$.
To do this, we use the equations for $G$ and $\Gamma$ to compute
\[
-\Delta \Big(  \sum_j  R(x,a_j) - S(x,a_j)\Big)
= - \Delta_x \Big( 2\pi \sum_j \Gamma(x,a_j) - G(x,a_j) \Big)
= 2\pi \sum_j G(x, a_j) . 
\]
Thus, applying $(-\Delta+1)$ to both sides, we 
find that $\sum_j R(\cdot, a_j) - S(\cdot, a_j)$ satisfies the
boundary-value problem that characterizes $B_1$.
Thus we have completed the proof.
\end{proof}


\section{Energy-minimizers in $\calA_\ep^N$}\label{EM.Localization:UB, LB}

The proposition below completes our argument;
once it is known, the proof of Theorem \ref{T.main} follows by
{\em exactly} the argument given in Section \ref{sec:PMT}.

Recall that  $\tep$ is a constant that appears in the definition of $\calA_\ep^N$ and
was fixed in the proof of Lemma \ref{L.uMepNstar}.

\begin{prop}\label{interior}
Assume that $(u_\ep,A_\ep)$ minimizes $GL_\ep$ in $\calA_\ep^N$
and that $A_\ep$ satisfies the
Coulomb gauge condition \eqref{coulomb}.
Then
\be\label{GLupper}
GL_\ep (u_\ep,A_\ep) \le \min_{M_{\ep,N}}H_\ep^N +\kappa_\ep^{GL}  +\frac \tep3 ,
\ee
and in addition,  there exists $\xi \in M_{\ep,N}$ such that
\be
\left\| Ju_\ep - \pi\sum_{i=1}^N \delta_{\xi_i} \right\|_{\dot W^{-1,1}}\le 
\frac 12 \sigma_\ep
\label{xi.vnear}\ee
and 
\be
GL_\ep(u_\ep,A_\ep) \ge  H_\ep^N(\xi) +\kappa_\ep^{GL}   - \frac \tep3 .
\label{Hxi.upper}\ee
\end{prop}


Before we can use results from the previous section
effectively, we need to control $E_\ep(u_\ep)$, which appears in many error terms.
This is the point of our first lemma.

\begin{lemma}\label{L.Eep}
Assume that $(u,A)$
satisfies \eqref{close} for some  $\sigma_\ep>0$ and $a\in \Omega^N$,
with $N\le \hex|\Omega|/2\pi$.
Then
\be\label{Leap}
E_\ep(u) \le GL_\ep(u,A) + C \hex^2 +C\ep\hex^4+ C \ep\, GL_\ep(u,A)^2 + 
C \sigma_\ep
(GL_\ep(u,A)+\hex^2)^{3/2} . 
\ee
for constants depending only on $\Omega$.\end{lemma}

\begin{proof} Recall from \eqref{GsplitB} the  identity
\[
E_\ep(u) = GL_\ep(u,\nabla^\perp B) - \left( \Phi(B) - 2\pi \sum_{i=1}^N B(a_i)
\right)
+ R(u,\nabla^\perp B) +  2\int_\Omega B\left(Ju  - \pi \sum \delta_{a_i} \right),
\]
where $R(u,A):= \frac 12 \int_\Omega (|u|^2-1)|A|^2$.
Arguing as in \eqref{BH2}, \eqref{need.soon}, we see that
$B = (\nabla^\perp)^{-1}A$ satisfies
\be\label{B.last}
\| B\|_{H^2}^2 \le  C (GL_\ep(u, A) + \hex^2).
\ee
As in \eqref{R.est}, we deduce that
\[
|R(u,A)|\le C\, \ep \, (GL_\ep(u,A)^2 + \hex^4).
\]
Next, by combining \eqref{enough} and \eqref{Jdiff} with \eqref{B.last} and a Sobolev embedding, we find that for any $m>2$,
\[
\left|\int_\Omega B\left(Ju  - \pi \sum \delta_{a_i} \right)\right|\le
C\sigma_\ep^{1-\frac 2m}(GL_\ep(u, A) + \hex^2)(E_\ep(u) + \pi N)^{\frac 2m} 
\]
Taking $m=6$ and using elementary inequalities, we deduce that the right-hand side is
bounded by
\[
\frac 13 E_\ep(u) + C\hex + C \sigma_\ep
(GL_\ep(u,A)+\hex^2)^{3/2}. 
\]
The lemma follows by combining these estimates with \eqref{Phi-sumB}.
 \end{proof}

\begin{proof}[Proof of Proposition \ref{interior}]

{\em Step 0.}
Our eventual aim  is to use Propositions \ref{P.gstab} and \ref{P.localization}
from the previous section, 
adjusting the parameters in our scaling assumptions both to 
arrange that the hypotheses are satisfied and to control the
error terms. We will go through
the choice of parameters rather carefully, to make it clear that 
there is nothing circular in our argument.
We recall the assumptions:
\[
0<\ep< \ep_0, 
\qquad\qquad
K_1\le  \hex \le k_1 \ep^{-1/4},
\]
\[
1\le N \le  \min
\left\{ \frac{\hex}{2\pi}(|\Omega|-\hex^{-1/4}) , k_2\ep^{-1/10}\hex^{-1/5} \right\}.
\]
For the definition \eqref{calA.def} of $\calA_\ep^N$,  we will
choose
\be\label{sigmaep.def}
\sigma_\ep = \ep^{99/100} \, \max \{ N^5 \hex^{1/2}, \hex^2\}
\ee
If $a\in M_{\ep,N}$ then it follows from Proposition \ref{main.prop}
that
\[
\rho_a \ge c_1\hex^{-1/2}.
\]
Finally, we will only apply Propositions \ref{P.gstab} and
\ref{P.localization} to $(u,A)$ such that
\be\label{C3}
E_\ep(u) \le C_3\hex^2
\ee
for $C_3(\Omega)$ to be determined below.
(In fact we will take $C_3 = \max\{ C_5, C_6\}$, where these
constants are identified below.)
We have already imposed conditions on $K_1$ for example.
We now adjust $\ep_0, k_1, k_2$ as follows.

First, by decreasing $\ep_0$ and $k_1$ as necessary,
\be\label{K2}
\begin{aligned}
(\mbox{right-hand side of \eqref{p2.c1}}) &\le \frac 12 \sigma_\ep \quad&\mbox{ if }\widetilde \Sigma^{GL}_\ep \le \tep\mbox{ and $u$ satisfies \eqref{C3}}, \\
\varpi(u) &\le \frac \tep6\qquad \quad &\mbox{ if $u$ satisfies \eqref{C3}}.
\end{aligned}
\ee
Similarly, by decreasing $k_2$ we may assume that
\be
\mbox{ hypotheses \eqref{loc.h1}, \eqref{Ebds} of Proposition \ref{P.localization} hold}, \quad\mbox{ if
$u$ satisfies \eqref{C3} . }
\ee
Hypothesis \eqref{gstab.h1} of Proposition \ref{P.gstab} involves both an upper and lower bound on
$\sigma_\ep$. The lower bound $\sigma_\ep \ge \ep\sqrt{\log (\rho_0/\ep)}$ 
follows from the form of $\sigma_\ep$, after possibly adjusting $\ep_0$, and the upper bound
is less stringent than \eqref{loc.h1}, already satisfied.
Finally, we claim that by further decreasing $\ep_0$,
we can arrange that $N^5\rho_a^{-1}(\sigma_\ep+\ep E_\ep(u))$ is as
small as we like, and hence, in view of \eqref{K2}, that 
\be\label{e0last}
(\mbox{right-hand side of \eqref{gstab.c1}}) \ge H_\ep^N(a) +\kappa_\ep^{GL} - \frac \tep3
\qquad \mbox{ if  $u$ satisfies \eqref{C3}}.
\ee
To see this, note that
\begin{align*}
N^5\rho_a^{-1}\sigma_\ep \le
C N^5 \hex^{1/2}\sigma
&\le
C \ep^{99/100}\max\{ N^{10}h, N^5 h^{5/2} \} 
\end{align*}
Using $N \le k_2 (\ep\hex^2)^{-1/10}$ for large $\hex$ and $N \le C \hex$ for small $\hex$,
we deduce that
\[
N^5\rho_a^{-1}\sigma_\ep \le \begin{cases}
 C(k_2) \max \{ \ep^{-1/100}h^{-1}, \ep^{49/100}h^{3/2}\} &\mbox{ if }\ep^{-1/12}\le \hex \le \ep^{-1/4}\\
C \ep^{99/100} \hex^{11}&\mbox{ if }\hex\le \ep^{-1/12}
\end{cases}
\]
which can be made as small as we like by a suitable choice of $\ep_0$. 
Similar considerations show that the same holds for $N^5\rho^{-1}\ep E_\ep(u)$,
subject to \eqref{C3}. Thus we may achieve \eqref{e0last}.

{\em Step 1}. 
We now prove \eqref{GLupper}.
We start by noting that for every $N\ge 1$ satisfying \eqref{scaling2},
\be\label{Hneg}
\min_{a\in M_{\ep,N}}H_\ep^N(a) = \min_{a\in \Omega^N} \uH_\ep^N \le 0.
\ee
The  equality on the left is clear from 
Lemma  \ref{L.uMepNstar}, which, for this range of $N$, 
implies that  $\min_{\Omega^N} \uH_\ep^N$ is attained in $M_{\ep, N}$,
in which $\uH_\ep^N = H_\ep^N$.
The inequality follows by an easy induction argument, which relies on the
fact that $\xi_0$ and  $v_\ep$ vanish on $\partial \Omega$,
as does $G(\cdot, a)$, for every $a\in \Omega$.

Let $a$ minimize $H_\ep^N$ in $M_{\ep,N}^*$.
Then we deduce from \eqref{WHrelation}, \eqref{Phi-sumB} and \eqref{Hneg} that 
\be\label{C7}
W^N(a)\le  C_4 \hex^2, \qquad\qquad C_4 = C_4(\Omega).
\ee 
According to Lemma 14 in \cite{JSp2},  
there exists $u_a\in H^1(\Omega)$
such that 
\be\label{Jacua}
\left\| Ju_a - \pi\sum_{i=1}^N \delta_{a_i} \right\|_{W^{-1,1}}\le CN\ep\left( 1 + \ep N^2\rho_{a}^{-2}\right)
\ee
and
\[ 
\Sigma_\ep^{BBH}(u_a,a) = 
E_\ep(u_a)   - \kappa_\ep^{GL} - W^N(a) \le  CN\ep^2\rho_{a}^{-2} .
\] 
In particular, it follows from the above and \eqref{C7} that
\be\label{C8}
E_\ep(u_a)\le C_5\hex^2
\ee
for $C_5$ depending only on $\Omega$.
It then follows from \eqref{K2} that 
\be
\varpi(u_a)\le \frac \tep{6}.
\label{varpiu}
\ee

Now let $A_a$ minimize $A\mapsto GL_\ep(u_a,A)$
among all competitors satisfying the Coulomb gauge condition \eqref{coulomb}.
It follows that for $\ep<\ep_0$ with $\ep_0$ is small enough, then
\[ 
(u_a, A_a)\in \calA_\ep^N, \qquad\qquad \Sigma_\ep^{BBH}(u_a,a)\le  \frac\tep 6
\] 
Then we infer from the above and  Lemma \ref{L.reduction} that
the minimizer $(u_\ep,A_\ep)$ satisfies
\begin{align*}
GL_\ep(u_\ep,A_\ep) 
-\min_{M_{\ep,N}}H_\ep^N - \kappa_\ep^{GL}&\le
GL_\ep(u_a,A_a) 
-\min_{M_{\ep,N}}H_\ep^N - \kappa_\ep^{GL} \\
&=
GL_\ep(u_a,A_a) 
-H_\ep^N(a) - \kappa_\ep^{GL}
\\
&=
\Sigma_\ep^{GL}(u_a, A_a, a)
\\ 
&\le \Sigma_\ep^{BBH}(u_a, a) +\varpi (u_a)\le \frac \tep{3}\, ,
\end{align*}
proving \eqref{GLupper}.

{\em Step 2}. From \eqref{GLupper}, \eqref{Hneg}, and \eqref{Leap}, we see that
\[
E_\ep(u_\ep)\le C_6 \hex^2.
\]
Since $u_\ep\in \calA_\ep^N$, there exists 
$a_\ep\in M_{\ep, N}^*$ such that \eqref{close} holds.
We have arranged above that the other hypotheses \eqref{loc.h1},
\eqref{Ebds} of Proposition \ref{P.localization} hold.
Also, it follows from Step 1 above that
\[
\widetilde \Sigma_\ep^{GL}(u_\ep, A_\ep, a_\ep)
=
\Sigma_\ep^{GL}(u_\ep, A_\ep, a_\ep)
+\varpi(u_\ep)
\le \frac \tep2.
\]
Thus Proposition \ref{P.localization} (see in particular \eqref{p2.c1})
and \eqref{K2} yield $\xi\in \Omega^N$ such that 
\eqref{xi.vnear} is satisfied.
It follows from this and \eqref{close}
that 
\be
\|\pi \sum_{i=1}^N( \delta_{a_i} - \delta_{\xi_i})\|_{\dot W^{-1,1}} \le \frac 32 \sigma_\ep \overset{\eqref{sigmaep.def}} \le  \frac \pi2\hex^{-1/3} \quad\mbox{ for }0<\ep <\ep_0,
\label{xineara}\ee
and this and \eqref{contain0} imply that 
\be
\xi\in M_{\ep, N}, \qquad\mbox{ or in other words $d(\xi_i)  = \dist(\xi_i, \partial \Omega)\ge \hex^{-1/3}$ for every $i$.}
\label{xi.bdy}\ee
To see this, assume toward a contradiction that
$d(\xi_i)< \hex^{-1/3}$ for some $i$,
and define
\[
f(x) := \begin{cases}
0&\mbox{ if }d(x)\le \hex^{-1/3}, \\
d(x) - \hex^{-1/3} &\mbox{ if }\hex^{-1/3} \le 2\hex^{-1/3}, \\
\hex^{-1/3}&\mbox{ if }d(x)\ge 2\hex^{-1/3} . 
\end{cases}
\]
It follows from \eqref{contain0} and the assumption that $d(\xi_i)< \hex^{-1/3}$ for some $i$,
\[
\pi \int_\Omega f(x) \sum_i \delta_{a_i} = N \pi \hex^{-1/3}, \qquad
\pi \int_\Omega f(x) \sum_i \delta_{\xi_i} \le  (N-1) \pi \hex^{-1/3}.
\]
On the other hand,  from \eqref{xineara} and the construction of $f$, it is evident that
\[
\left|\pi \int_\Omega f(x) \sum_i (\delta_{a_i} - \delta_{\xi_i})\right| \le 
\|\pi \sum_{i=1}^N( \delta_{a_i} - \delta_{\xi_i})\|_{\dot W^{-1,1}} 
\| f\|_{\dot W^{1,\infty}} \le \frac \pi 2 \hex^{-1/3}.
\]
(Here we are using the same convention as in \cite{JSp2}, which is that  
$\| f\|_{\dot W^{1,\infty}} =  \|\nabla f\|_\infty\|$, for $f$ vanishing on $\partial \Omega$.)
This is a contradiction, proving \eqref{xi.bdy}, and completing the proof of 
\eqref{xi.vnear}.

{\em Step 3}. Finally, \eqref{Hxi.upper} follows immediately from Proposition
\ref{P.gstab}, in view of \eqref{e0last}. (Note that the hypotheses of the proposition
are satisfied due to \eqref{xi.vnear} and the choice of $\ep_0, k_1 $ {\em etc.}.)

\end{proof}


\end{document}